\documentclass[letterpaper,11pt]{article}

\usepackage[english]{babel}
\usepackage{fullpage}
\usepackage{amsmath,amssymb,amsthm,mathtools}
\usepackage{graphicx}
\usepackage{float}
\usepackage[footnotesize]{caption}
\usepackage[colorlinks=true,allcolors=blue]{hyperref}
\usepackage{doi}

\title{Cycle-Structure Generating Functions for Special Breakpoint Graphs}
\author{Max A. Alekseyev\thanks{The George Washington University, Washington, DC. Email: \href{mailto:maxal@gwu.edu}{maxal@gwu.edu}} \and Joseph T. Iosue\thanks{Joint Center for Quantum Information and Computer Science and Joint Quantum Institute, NIST/University of Maryland, College Park, MD, 20742, USA} \thanks{Present Address: Microsoft Quantum, Redmond, WA 98052, USA} \and Adam Ehrenberg$^\dagger$ \and Alexey V. Gorshkov$^\dagger$}
\date{}

\hypersetup{
  pdftitle={Cycle-Structure Generating Functions for Special Breakpoint Graphs},
  pdfauthor={Max A. Alekseyev and Joseph T. Iosue and Adam Ehrenberg and Alexey V. Gorshkov}
}

\newcommand{\G}{{\cal G}}
\newcommand{\C}{{\cal C}}
\newcommand{\coef}[1]{\mathop{[\![#1]\!]}}
\newtheorem{thm}{Theorem}
\newtheorem{lemma}[thm]{Lemma}
\newtheorem{corol}[thm]{Corollary}
\newtheorem{prop}[thm]{Proposition}
\theoremstyle{remark}
\newtheorem*{remark}{Remark}
\theoremstyle{plain}

\begin{document}
\maketitle

\begin{abstract}
Breakpoint graphs originate in comparative genomics, where their alternating
cycles encode relationships between genomes.  We study a constrained class of
three-colored breakpoint graphs associated with permutations and develop
cycle-refined generating functions for two extremal families.  These families
have a natural topological interpretation: their canonical surfaces are,
respectively, the sphere and the projective plane.  The spherical family is
characterized by noncrossing configurations, while the projective-plane
family admits a different decomposition involving a distinguished family of
M\"obius ladders.

The resulting generating-function equations retain the full cycle structure
but nevertheless admit substantial reductions.  This leads to explicit
Catalan-weighted evaluations, polynomiality results for refined cycle
statistics, and a connection between a natural diagonal specialization and
noncrossing trees.  The two topological families exhibit markedly different
combinatorial mechanisms, providing complementary examples of how local
transformations of breakpoint graphs can control refined permutation
enumerations.

As a further application, the same Catalan-weighted sums arise in asymptotic
unitary-Weingarten expansions for entanglement of random Gaussian states in
linear optics.  The combinatorial results determine the leading and
constant-order moment polynomials entering the R\'enyi entropy expansion,
with the projective-plane contribution giving the finite-size constant
correction.
\end{abstract}

\section{Introduction}\label{sec:intro}
A breakpoint graph is the union of three perfect matchings on the same vertex set. Every pair of colors decomposes into alternating cycles, and the interaction among the three cycle systems records rich combinatorial information. Breakpoint graphs originated in comparative genomics, where their cycle structure encodes relationships between genome arrangements and underlies rearrangement-distance statistics. Multivariate enumeration of breakpoint-graph cycle structures, including generalized Hultman numbers, was developed in \cite{Alexeev2017hultman}; related noncrossing and topological structures appear in \cite{Alekseyev2007,Alexeev2016top}.

Here we study a more constrained class, which we call \emph{special breakpoint graphs}. Throughout the paper we call the three matchings, and their edges, \emph{black}, \emph{gray}, and \emph{red}; the formal definition is given in Section~\ref{sec:breakpoint}. Two of the matchings form one rooted alternating cycle, while the third matching obeys a parity condition around that cycle. Equivalently, these graphs encode permutations in $S_{2n}$ while simultaneously recording the permutation-cycle decomposition and an additional alternating-cycle statistic. Our goal is to develop the intrinsic combinatorics of this class: structural characterizations, local graph transformations, and multivariate generating functions that retain the complete black-gray cycle structure.

The combinatorial parameter organizing the paper is
\begin{equation}\label{eq:kappa_intro}
\kappa:=\bigl(\#\text{black-gray cycles}\bigr)+\bigl(\#\text{gray-red cycles}\bigr)-2n.
\end{equation}
A canonical surface construction in Section~\ref{sec:breakpoint} shows that $\kappa+1$ is the Euler characteristic of a closed connected surface carried by the three-colored graph. Since every closed connected surface has Euler characteristic at most $2$, necessarily $\kappa\leq1$. The maximal stratum $\kappa=1$ is spherical, while the next stratum $\kappa=0$ is carried by the projective plane. This topological distinction anticipates their different recursive structures and is the reason we focus on these two strata. For $\kappa=1$, the graphs are exactly the noncrossing special breakpoint graphs; cutting at the root produces an ordered triple of smaller noncrossing graphs and leads to a nonlinear differential equation for the multivariate generating function. For $\kappa=0$, every graph either contains a black-gray $1$-cycle or belongs to an exceptional odd M\"obius-ladder family; removing a rooted $1$-cycle gives a different recursive generating-function equation.

A second motivation comes from random Gaussian-state entanglement in linear optics. Page curves for Haar-random passive linear-optical transformations can be expressed through unitary-Weingarten permutation sums~\cite{Iosue2023,Youm2025}. In the extremal powers relevant to the first two orders of the large-system expansion, those sums are precisely signed Catalan-weighted special-breakpoint-graph sums in the strata $\kappa=1$ and $\kappa=0$. This is one reason the two topological strata and the Catalan weight sequence are singled out here. The mathematical development below is independent of this application. We return to the physical dictionary only in Section~\ref{sec:physics_application}, where the Catalan evaluations proved earlier are translated into the entanglement problem and compared with recent related work~\cite{AltPageCurves}, which carries the corresponding series to closed form by complementary methods.

The primary output of the paper is therefore mathematical: multivariate recurrences, structural descriptions, reduction mechanisms, and intrinsic weighted evaluations for the spherical and projective-plane classes. Section~\ref{sec:breakpoint} defines special breakpoint graphs and their cycle-structure generating functions and proves the basic lemmas used in both strata. Section~\ref{sec:kappa=1} then treats the spherical class as a single block: noncrossing structure, the recursive PDE, the universal shift-flow reduction, and the signed Catalan evaluation. Section~\ref{sec:kappa=0} treats the projective-plane class in the same way: M\"obius ladders and root removal, the recursive PDE and its integral form, the $\mathfrak{sl}_2$ reduction, the Catalan orbit and diagonal shift-flow, and finally the coefficient formulas at fixed gray-red cycle number and fixed permutation cycle type. Section~\ref{sec:physics_application} translates selected Catalan evaluations proved earlier in Sections~\ref{sec:kappa=1} and~\ref{sec:kappa=0} into the Gaussian-state entanglement problem, without adding new graph-theoretic arguments. Section~\ref{sec:extensions} summarizes the main conclusions and discusses further directions.

\begin{figure}[!t]
    \centering
    \includegraphics[width=.75\linewidth]{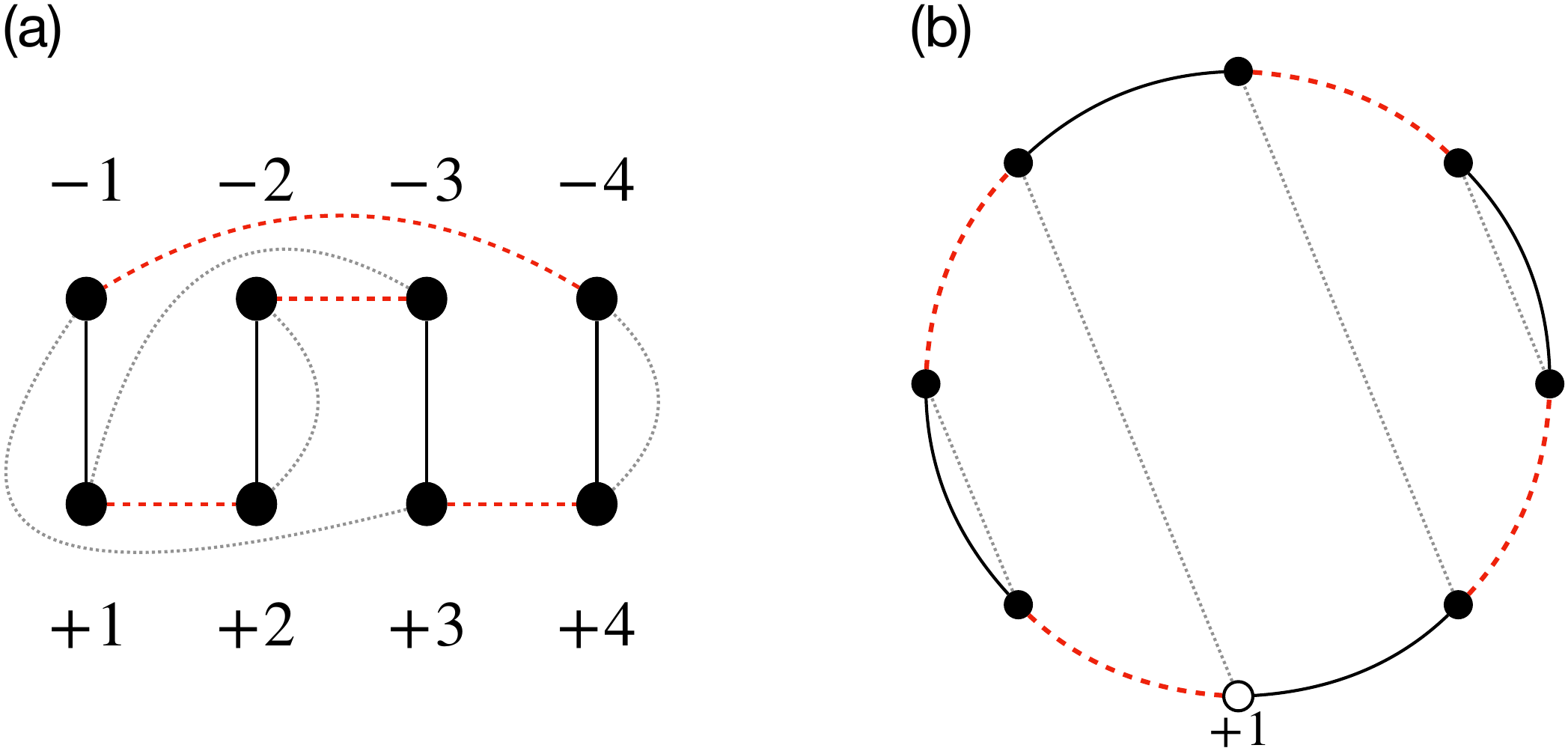}
    \caption{An example of graph representations of a permutation $\tau=(3,2,1,4)$ with $n=2$.  \textbf{(a)} The corresponding special breakpoint graph. \textbf{(b)} The same graph with the black-red cycle drawn as a circle and gray edges drawn as chords.  The marked vertex $+1$ is the root and determines the remaining labels.  This example is noncrossing.}
    \label{fig:correspondence}
\end{figure}

\section{Special breakpoint graphs and cycle-structure generating functions}\label{sec:breakpoint}

We rely on graphs formed by three perfect matchings, traditionally studied in comparative genomics under the name \emph{breakpoint graphs}~\cite{Alexeev2017hultman,Alekseyev2007}. We find it convenient to refer to the three matchings and their edges as colored \emph{black}, \emph{gray}, and \emph{red}. So, a breakpoint graph is a graph formed by $m$ black, $m$ gray, and $m$ red edges on the same $2m$ labeled vertices, where the edges of each color form a perfect matching. It is clear that the edges of every pair of colors form a collection of cycles, along which the colors of edges alternate. We will refer to such cycles as \emph{black-gray}, \emph{black-red}, and \emph{gray-red}, respectively. In computational biology, a breakpoint graph describes the relationship between two (circular) genomes on the same set of genes, where (say) red edges encode genes, black-red cycles represent chromosomes (as sequences of genes) in one genome, and gray-red cycles represent chromosomes in the other genome. The number of black-gray cycles is a central statistic in rearrangement-distance formulas.

The cycle structures of the breakpoint graphs with a single black-red cycle were enumerated in~\cite{Alexeev2017hultman} by means of multivariate generating functions, which further enabled a few applications (e.g., uniform sampling of genomes at a fixed distance) and revealed a connection to known combinatorial objects such as exponential Bell polynomials. In the present work, we consider \emph{special breakpoint graphs} on $4n$ vertices. For a permutation $\tau\in S_{2n}$, written in one-line notation unless stated otherwise, label the vertices by $\pm i$ for $i\in\{1,2,\dots,2n\}$; black edges connect $-i$ with $+i$, gray edges connect $-i$ with $+\tau(i)$, and red edges connect $(-1)^{i-1}i$ with $(-1)^{i-1}((i\bmod 2n)+1)$ (Fig.~\ref{fig:correspondence}a). Since the labeling of vertices in a special breakpoint graph is uniquely defined by the label of a single vertex (root), say $+1$, we will not label vertices but assume that the graph is rooted (Fig.~\ref{fig:correspondence}b).

The three cycle systems also give a canonical topological realization of a special breakpoint graph.  This interpretation will explain why $\kappa=1$ and $\kappa=0$ are the first two strata.

\begin{prop}[Canonical surface]\label{prop:canonical_surface}
Let $G$ be a special breakpoint graph of order $4n$, and let $c_{BG}$, $c_{GR}$, and $c_{BR}$ denote its numbers of black-gray, gray-red, and black-red cycles.  Then $G$ canonically determines a closed connected surface $\Sigma_G$ (not necessarily orientable) with a cellular embedding of $G$ whose faces are precisely the bicolored cycles.  Its Euler characteristic is
\begin{equation}\label{eq:kappa_euler_characteristic}
\chi(\Sigma_G)=c_{BG}+c_{GR}+c_{BR}-2n=\kappa(G)+1.
\end{equation}
Consequently $\kappa(G)\leq1$.  Moreover,
\[
\kappa(G)=1\quad\Longleftrightarrow\quad \Sigma_G\cong S^2,
\qquad
\kappa(G)=0\quad\Longleftrightarrow\quad \Sigma_G\cong\mathbb{RP}^2.
\]
\end{prop}

\begin{proof}
For each vertex of $G$, take a triangle whose three sides are labeled black, gray, and red.  Label each corner by the pair of colors of its incident sides.  If two vertices of $G$ are joined by an edge of a given color, glue the correspondingly colored sides of their triangles, matching corners with the same two-color label.  Every triangle side is paired exactly once, so the resulting two-dimensional cell complex has no boundary.  The link of a vertex labeled, for example, $BG$ traces one black-gray alternating cycle of $G$; hence every vertex link is a circle.  Thus the complex is a closed surface.  It is connected because the black-red matching pair already forms one cycle through all vertices of $G$.  Placing each vertex of $G$ at the center of its triangle and each colored edge across the corresponding glued side realizes $G$ as the dual cellular graph; the dual face around a triangulation vertex is exactly the associated bicolored cycle.

In this triangulation there are $4n$ triangular faces and $6n$ edges. Its vertices are in bijection with the three families of bicolored cycles, so their number is $c_{BG}+c_{GR}+c_{BR}$.  Therefore
\[
\chi(\Sigma_G)
=(c_{BG}+c_{GR}+c_{BR})-6n+4n.
\]
For a special breakpoint graph $c_{BR}=1$, and~\eqref{eq:kappa_euler_characteristic} follows.  A closed connected surface has Euler characteristic at most $2$; equality characterizes the sphere, while Euler characteristic $1$ characterizes the projective plane.  This proves the remaining assertions.
\end{proof}

The construction is already visible in the smallest spherical example.  For
$n=1$ and the identity permutation $\tau=(1,2)$, the gray matching coincides
with the black matching, so $c_{BG}=2$ and $c_{GR}=c_{BR}=1$.  Thus
$\chi(\Sigma_G)=2$.  Figure~\ref{fig:canonical_surface_example} shows the
passage from the special breakpoint graph to the corresponding four-face
triangulation and then to its realization on the sphere.

\begin{figure}[H]
    \centering
    \includegraphics[width=.98\linewidth]{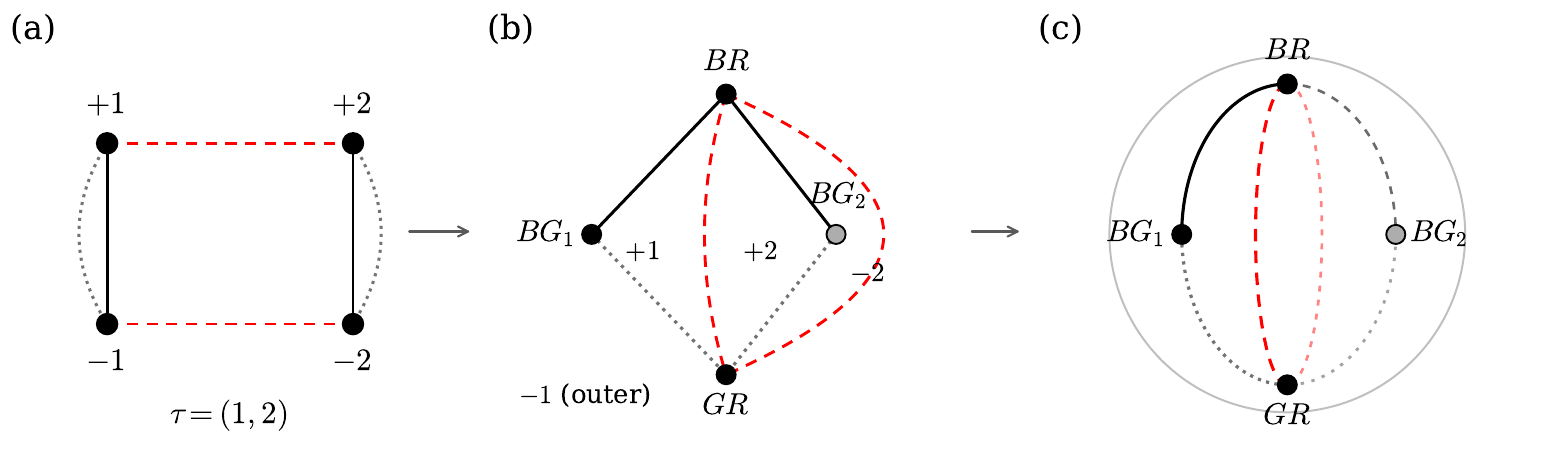}
    \caption{The canonical-surface construction in the smallest spherical
    example.  \textbf{(a)} The special breakpoint graph for $n=1$ and
    $\tau=(1,2)$.  The gray matching coincides with the black matching and is
    drawn by offset dotted arcs so that both colors remain visible.
    \textbf{(b)} The $1$-skeleton of the associated four-face triangulation,
    drawn as a planar map; the region marked $-1$ is the outer face.  Its
    vertices $BG_1,BG_2,BR,GR$ correspond to the bicolored cycles, while its
    four triangular faces $+1,+2,-1,-2$ correspond to the four vertices of
    the breakpoint graph.  In particular, each red side joins $BR$ to $GR$.
    \textbf{(c)} The same triangulation realized on $S^2$.  All colored arcs
    lie on the sphere; dark arcs lie on the visible hemisphere and lighter
    broken arcs on the back hemisphere.  Hence the two red arcs meet the
    closed black/gray curve $BR-BG_1-GR-BG_2-BR$ only at $BR$ and $GR$.
    The rear vertex $BG_2$ is shaded gray.}
    \label{fig:canonical_surface_example}
\end{figure}

\begin{remark}[Polygon-gluing formulation]
The canonical surface can equivalently be viewed through the polygon-gluing language used for breakpoint graphs~\cite{Alekseyev2007,Alexeev2016top}.  Cutting the colored surface along one matching turns the relevant bicolored cycles into polygonal boundary components, and the third matching prescribes how their sides are identified.  Conversely, the alternating cycles created by those identifications are precisely the cycle data entering the Euler characteristic.  Thus the polygon-gluing and colored-triangulation descriptions are two cellulations of the same matching-to-surface mechanism; the latter is convenient here because it treats all three colors symmetrically and gives $\kappa$ an intrinsic topological meaning before any generating-function specialization.
\end{remark}

\begin{figure}[H]
    \centering
    \includegraphics[width=\linewidth]{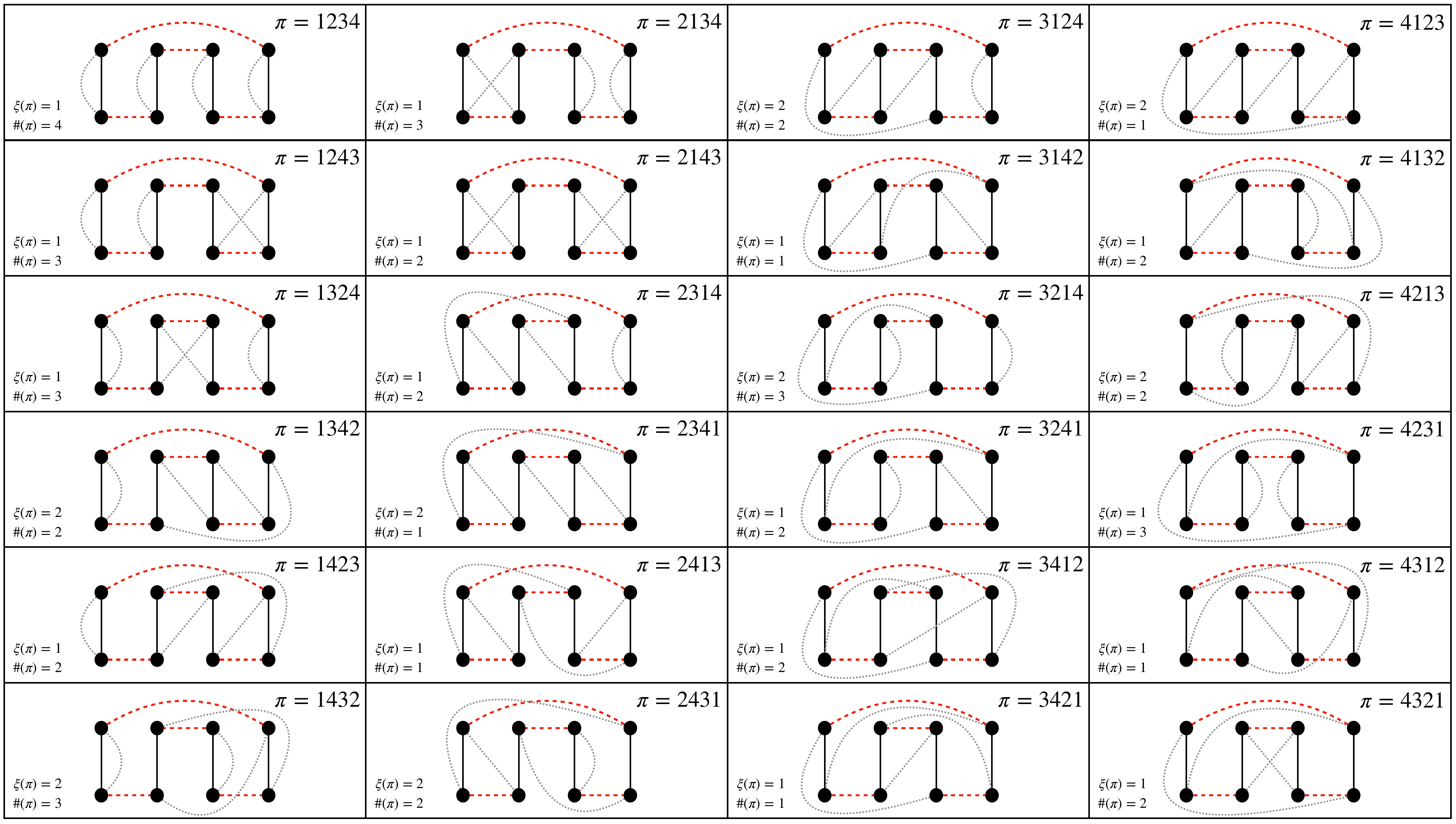}
    \caption{Example for $n=2$ (so $2n=4$). 
    The 24 diagrams are the breakpoint graphs associated with the permutations $\pi \in S_{4}$ (here written in one-line notation). $\xi(\pi)$ denotes the number of gray-red cycles. $\#(\pi)$, the number of cycles of the permutation, is the number of black-gray cycles.}
    \label{fig:example}
\end{figure}

\begin{remark}[Permutation-length form]
Let $\tau\in S_{2n}$ be the permutation encoded by a special breakpoint graph, let $\operatorname{cyc}(\tau)$ denote its set of cycles, let $\#(\tau):=|\operatorname{cyc}(\tau)|$, and define its transposition length by
\[
|\tau|:=2n-\#(\tau).
\]
Let $\xi(\tau)$ denote the number of gray-red cycles. Since the black-gray cycles are exactly the cycles of $\tau$, we have $c_{BG}=\#(\tau)$ and $c_{GR}=\xi(\tau)$, so the stratum parameter has the equivalent form
\begin{equation}\label{eq:kappa_transposition_length}
\kappa(\tau)=\xi(\tau)-|\tau|.
\end{equation}
Thus the spherical and projective-plane strata may also be written as
\[
\xi(\tau)=|\tau|+1\qquad(\kappa=1),
\qquad
\xi(\tau)=|\tau|\qquad(\kappa=0).
\]
Figure~\ref{fig:example} illustrates this dictionary for all permutations in $S_4$, displaying $\xi(\pi)$ and $\#(\pi)$ alongside their breakpoint graphs. This permutation-theoretic form is particularly convenient for the fixed-support arguments later in Section~\ref{sec:kappa=0}, while Proposition~\ref{prop:canonical_surface} gives the same parameter its topological meaning.
\end{remark}

The canonical-surface bound makes $\kappa=1$ and $\kappa=0$ the top two strata. We treat them separately, because their topology leads to different recursive decompositions, in Theorems~\ref{th:G1_PDE} and~\ref{th:G0_PDE}, respectively.

For nonnegative integers $n,c_1,c_2,\dots$, we define $g_\kappa(n,c_1,c_2,\dots)$ as the number of graphs formed by three perfect matchings (black, gray, red) on $4n$ unlabeled vertices, one of which is marked as the root, such that
\begin{itemize}
    \item black and red edges form a single cycle.  Starting at the root, labeled $+1$, label the vertices by $\pm i$ for $1\le i\le 2n$ so that black edges are $\{-i,+i\}$ and red edges join $(-1)^{i-1}i$ to $(-1)^{i-1}((i\bmod 2n)+1)$;
    \item endpoints of each gray edge have labels of different signs;\footnote{Equivalently, if we consecutively label the red edges along the black-red cycle with numbers $1,2,\dots, 2n$, then each gray edge connects an odd-numbered red edge to an even-numbered one.}
    \item the number of black-gray \emph{$i$-cycles} (i.e., with $i$ black and $i$ gray edges) is $c_i$;
    \item the total number of gray-red and black-gray cycles is $2n+\kappa$, i.e., the number of gray-red cycles is 
    $$2n + \kappa - c_1 - c_2 - \dots.$$
\end{itemize}
We further define the following cycle-structure generating functions:
\[
\begin{split}
\G_{\kappa}(x,u,s_1,s_2,\dots) &:= \mathbf 1_{\{\kappa=1\}}u + \sum_{n\geq 1} \frac{x^n}{2n} \sum_{c_1 + 2c_2 + \dots = 2n} g_\kappa(n,c_1,c_2,\dots)  u^{2n+\kappa - c_1 - c_2 - \dots}  s_1^{c_1} s_2^{c_2} \cdots \\
&= \mathbf 1_{\{\kappa=1\}}u + u^{\kappa} \sum_{n\geq 1} \frac{(xu^2)^n}{2n} \sum_{c_1 + 2c_2 + \dots = 2n} g_\kappa(n,c_1,c_2,\dots)  \left(\frac{s_1}u\right)^{c_1} \left(\frac{s_2}u\right)^{c_2} \cdots .
\end{split}
\]

The extra term $u$ for $\kappa=1$ represents the empty graph; it is convenient for the recursive decomposition in Section~\ref{sec:kappa=1} and does not affect any positive-degree coefficient. Thus $x$ marks half the permutation size, $u$ marks gray-red cycles, and $s_i$ marks black-gray $i$-cycles. For a formal power series $F$, we write $\coef{z^m}F$ for the coefficient of $z^m$ in $F$.

We use comma subscripts for partial derivatives throughout; for example, $\G_{\kappa,x}:=\partial \G_\kappa/\partial x$ and $\G_{\kappa,s_i}:=\partial \G_\kappa/\partial s_i$. Lowercase letters used as vertex labels in local constructions and figures are unrelated to the generating variables $u,s_1,s_2,\ldots$.

\begin{remark}[Rooted-size normalization]
The factor $1/(2n)$ in the definition of $\G_\kappa$ is chosen to match the rooted decompositions used below. Coefficientwise it is removed by differentiation:
\[
2\frac{\partial}{\partial x}\left(\frac{x^n}{2n}\right)=x^{n-1}.
\]
This is why the $\kappa=1$ decomposition naturally produces $2\G_{1,x}$, and the same normalization is retained for $\kappa=0$ so that the two series use a common convention.
\end{remark}

\subsection{Rooted cycle selection and a switching lemma}

For nonnegative integers $n,c_1,c_2,\dots$ and an integer $i\geq 1$, let $g^{(i)}_\kappa(n,c_1,c_2,\dots)$ be the number of special breakpoint graphs of order $4n$ with black-gray cycle structure $(c_1,c_2,\dots)$ such that the root belongs to a black-gray $i$-cycle. We will find it convenient to define generating functions:
$$
\G^{(i)}_{\kappa}(x,u,s_1,s_2,\dots) := \sum_{n\geq 1} x^n \sum_{c_1 + 2c_2 + \dots = 2n} g^{(i)}_\kappa(n,c_1,c_2,\dots)  u^{2n+\kappa - c_1 - c_2 - \dots}  s_1^{c_1} s_2^{c_2} \cdots.
$$

We will need the following lemmas:

\begin{lemma}\label{lem:gfi}
For any $\kappa\in\{0,1\}$ and any integers $n\geq 1$ and $i\geq 1$,
$$
g^{(i)}_\kappa(n,c_1,c_2,\dots) = \frac{ic_i}{2n} g_\kappa(n,c_1,c_2,\dots),
$$
and correspondingly
$$
\G^{(i)}_{\kappa}(x,u,s_1,s_2,\dots) = i s_i \frac{\partial}{\partial s_i} \G_{\kappa}(x,u,s_1,s_2,\dots).
$$
\end{lemma}

\begin{proof}
Let $V_i(G)$ be the set of vertices in the black-gray $i$-cycles of a special breakpoint graph $G$, and so $|V_i(G)|=2ic_i$.
Consider the dihedral re-rooting group $D_{2n}$ of order $4n$ acting on the special breakpoint graphs of order $4n$ (and thus on subsets of their vertices).
Let $A(G)\leq D_{2n}$ be the automorphism group of the unrooted version of $G$, for which $V_i(G)$ forms an invariant. This action is free on vertices: a color-preserving automorphism that fixes one vertex fixes its unique black and red neighbors, and hence fixes the entire black-red cycle. Therefore every $A(G)$-orbit of vertices has size $|A(G)|$.
It follows that $|D_{2n}\cdot G| = 4n/|A(G)|$, and the proportion of graphs in $D_{2n}\cdot G$ with the root in $V_i(G)$ equals $\frac{|V_i(G)|/|A(G)|}{4n/|A(G)|} = \frac{ic_i}{2n}$.

The displayed proportion depends only on the cycle structure, so summing over all unrooted graphs with the prescribed cycle structure gives
\[
g^{(i)}_\kappa(n,c_1,c_2,\dots)=\frac{ic_i}{2n}g_\kappa(n,c_1,c_2,\dots).
\]
Multiplying the defining series of $\G_\kappa$ by $ic_i$ is exactly the action of $i s_i\partial/\partial s_i$, which proves the generating-function identity.
\end{proof}

\begin{lemma}\label{lem:2break}
Let $G$ be a special breakpoint graph, and suppose that a special breakpoint graph $G'$ is obtained from $G$ by deleting two gray edges and re-pairing their four endpoints in a different way. Then the number of black-gray cycles changes by $\pm1$. More precisely, if the two deleted gray edges lie in distinct black-gray cycles, those cycles merge into one; if they lie in the same black-gray cycle, that cycle splits into two.
\end{lemma}

\begin{proof}
Number the red edges consecutively along the black-red cycle by $1,2,\ldots,2n$, and call a vertex odd or even according to the parity of its incident red edge. Every black edge joins an odd vertex to an even vertex, and the defining parity condition for a special breakpoint graph says the same for every gray edge.

Write the two deleted gray edges as $\{o_1,e_1\}$ and $\{o_2,e_2\}$, where the $o_i$ are odd and the $e_i$ are even. Of the two nontrivial pairings of these four endpoints, the pairing $\{o_1,o_2\},\{e_1,e_2\}$ violates the parity condition. Hence the only re-pairing that can again give a special breakpoint graph is
\[
\{o_1,e_2\},\qquad \{o_2,e_1\}.
\]

If the deleted gray edges lie in distinct black-gray cycles, deleting them opens the two cycles into two alternating paths, each with one odd and one even endpoint. The admissible cross-pairing above joins the two paths into a single cycle. If the deleted gray edges lie in the same black-gray cycle, deleting them produces two alternating paths, again each with one odd and one even endpoint; the same admissible re-pairing closes the two paths separately. Thus a pair of cycles merges in the first case and one cycle splits in the second, so the number of black-gray cycles changes by $\pm1$.
\end{proof}

\begin{remark}[Equivalent forms of the switching rule]
There is an equivalent geometric way to see the forbidden reconnection. Orient every black edge by the orientation inherited from the black-red Hamiltonian cycle. When a black-gray cycle is traversed, all of its black edges are followed consistently with these orientations. If two gray edges on one black-gray cycle were reconnected so as to produce a single ``twisted'' cycle rather than two cycles, one of the two intervening alternating paths would be reversed, and its black edges would then be traversed in the opposite direction. Hence such a twist cannot occur. The parity argument in Lemma~\ref{lem:2break} is the local endpoint version of this orientation obstruction.

In the usual breakpoint-graph language, the same move is a local double-cut-and-join (DCJ) surgery: two gray matching edges are cut and their four endpoints are paired again.  The special parity condition leaves exactly one admissible nontrivial re-pairing.  Thus the parity, orientation, and DCJ descriptions are three views of the same split--merge operation.
\end{remark}

\section{The spherical class \texorpdfstring{$\kappa=1$}{kappa=1}}\label{sec:kappa=1}
The spherical stratum is the extremal class: its canonical surface is a sphere, and its chord representation is noncrossing. We first turn this structure into a recursive equation for the full cycle-refined generating function, and then show that the resulting infinite-variable flow collapses canonically to one parameter, with the signed Catalan specialization becoming explicit.

\subsection{Noncrossing structure and the recursive equation}\label{sec:kappa1_structure}

The case of $\kappa=1$ corresponds to the special breakpoint graphs of order $4n$ with the total number of black-gray and gray-red cycles equal $2n+1$. The generating function in this case starts with
\[
\begin{split}
\G_1(x,u,s_1,s_2,\dots) & = u \\
&+ s_1^2 u \frac{x}{2} \\
&+ \big(2 s_1^2 s_2 u^2 + s_1^4 u\big) \frac{x^2}4 \\
&+ \big((3 s_1^2 s_2^2 + 2 s_1^3 s_3) u^3 + 6 s_1^4 s_2 u^2 + s_1^6 u\big) \frac{x^3}6 \\
&+ ((4 s_1^2 s_2^3 + 8 s_1^3 s_2 s_3 + 2 s_1^4 s_4) u^4 + (20 s_1^4 s_2^2 + 8 s_1^5 s_3) u^3 + 12 s_1^6 s_2 u^2 + s_1^8 u) \frac{x^4}8 \\
&+ \dots
\end{split}
\]

For example, in the coefficient of $\frac{x^2}4$ the term $2 s_1^2 s_2 u^2$ enumerates permutations $(3,2,1,4)$ and $(1,4,3,2)$, while the term $s_1^4 u$ enumerates the identity permutation (Fig.~\ref{fig:example}).

Figure~\ref{fig:correspondence} shows that the special breakpoint graph of the permutation $\tau=(3,2,1,4)$ with $\kappa=1$ is \emph{noncrossing}, that is its gray edges drawn as straight chords within its black-red cycle drawn as a circle do not cross. In fact, this is more than just a coincidence but rather a characterizing property of the special breakpoint graphs in the case of $\kappa=1$:

\begin{lemma}\label{lem:noncrossing}
For a positive integer $n$, let $G$ be a special breakpoint graph of order $4n$.
The total number of black-gray and gray-red cycles in $G$ equals $2n+1$ if and only if $G$ is noncrossing.
\end{lemma}

\begin{proof}
By Proposition~\ref{prop:canonical_surface}, the equality $c_{BG}+c_{GR}=2n+1$ is equivalent to $\chi(\Sigma_G)=2$, hence to $\Sigma_G\cong S^2$.  In the canonical embedding, the unique black-red cycle is a Hamiltonian face.  Removing the interior of that face from the sphere leaves a disk whose boundary is the black-red cycle and which contains all gray edges.  The gray edges are therefore disjoint properly embedded arcs with endpoints on the boundary.  Their endpoint pairing is noncrossing, and after an isotopy of the disk the arcs may be drawn as nonintersecting straight chords.

Conversely, suppose the gray edges can be drawn as noncrossing chords inside the black-red cycle.  Together with the black-red boundary edges this gives a planar embedding with the same local color order as the canonical colored-triangulation embedding.  Capping the exterior of the black-red cycle by a disk yields the canonical surface $\Sigma_G$ as a sphere.  Proposition~\ref{prop:canonical_surface} then gives $\kappa=1$, equivalently $c_{BG}+c_{GR}=2n+1$.
\end{proof}

\begin{remark}[Equivalent extremal formulations]
The same extremal structure has three useful descriptions.  In the simple BG-graph formulation~\cite{Alekseyev2007}, recolor every red edge of $G$ black and double every gray edge.  The resulting graph has its black edges on one Hamiltonian circle, while its gray edges are doubled chords, and its number of black-gray alternating cycles is $c_{BG}+c_{GR}$.  Hence the conditions $c_{BG}+c_{GR}=2n+1$, maximal alternating-cycle count, noncrossing chord representation, and $\Sigma_G\cong S^2$ are equivalent.  They are cycle-theoretic, chord-diagram, and surface formulations of the same phenomenon; the proof above simply uses the surface language.
\end{remark}

\begin{thm}\label{th:G1_PDE}
The function $\G_1(x,u,s_1,s_2,\dots)$ satisfies the following PDE:
$$
2 \frac{\partial \G_1}{\partial x} = \big( s_1 + \sum_{i\geq 1} i s_{i+1} \frac{\partial \G_1}{\partial s_i}\big)^2 \big(u + 2x\frac{\partial \G_1}{\partial x}\big)
$$
with the boundary condition $\G_1(0,u,s_1,s_2,\dots) = u$.
\end{thm}

\begin{proof} By Lemma~\ref{lem:noncrossing}, the special breakpoint graphs enumerated by $\G_1$ are exactly the noncrossing breakpoint graphs.

Let $S$ be the set composed of the noncrossing special breakpoint graphs and the empty graph. 
We define a bijective transformation of graphs from $S$ of order $4n$ for each $n>0$ into an ordered triple of graphs from $S$ of the total order $4(n-1)$ as follows. For a graph $G\in S$ rooted at vertex $a$, let $\{a,b\}$ and $\{a,c\}$ be the red and gray edges incident to $a$, and further let $(u,a,c,v)$ and $(x,b,d,y)$ be black-gray paths, and $\{c,w\}$ and $\{d,z\}$ be red edges (Fig.~\ref{fig:k1transform}a). By placing $G$ in the plane without edge crossings, doubling gray edges $\{a,c\}$ and $\{b,d\}$ and cutting $G$ along them, we split $G$ into three parts sharing only these gray edges and further contract them into three graphs from $S$ as follows (Fig.~\ref{fig:k1transform}b):

\begin{figure}[H]
    \centering
    \includegraphics[width=\linewidth]{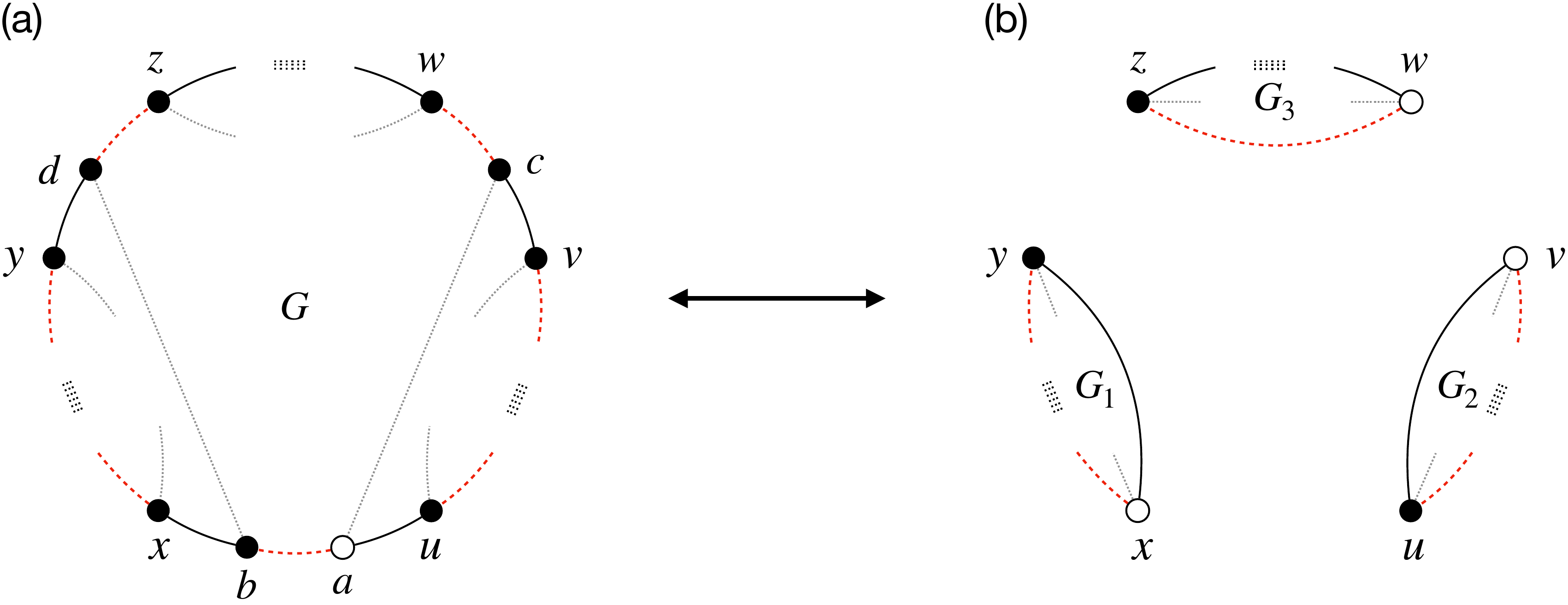}
    \caption{A bijective transformation of \textbf{(a)} a noncrossing special breakpoint graph of order $4n$ rooted at vertex $a$ into \textbf{(b)} three noncrossing special breakpoint graphs rooted at vertices $x,v,w$ and the total order $4(n-1)$.}
    \label{fig:k1transform}
\end{figure}
\begin{itemize}
\item $G_1$ rooted at $x$ is obtained from the part containing black-gray path $(x,b,d,y)$ by contracting it into a black edge $(x,y)$, except for the case $(b,x)=(y,d)$ when $G_1$ is empty;
\item $G_2$ rooted at $v$ is obtained from the part containing black-gray path $(u,a,c,v)$ by contracting it into a black edge $(u,v)$, except for the case $(a,u)=(v,c)$ when $G_2$ is empty;
\item $G_3$ rooted at $w$ is obtained from the part containing red-gray path $(w,c,a,b,d,z)$ by contracting it into a red edge $(w,z)$, except for the case $(d,z)=(w,c)$ when $G_3$ is empty.
\end{itemize}

The constructed bijection implies that the black-gray cycle structure of $G$ is the same as that of $G' := G_1 \cup G_2\cup G_3$, except that one $i$-cycle and one $j$-cycle in $G'$ are replaced in $G$ by an $(i+1)$-cycle and a $(j+1)$-cycle, respectively, where $2i$ and $2j$ are the lengths of the black-gray cycles containing the black edges $(x,y)$ and $(u,v)$ in $G'$. In the exceptional cases of $G_1$ and/or $G_2$ being empty, the black-gray cycle structure of $G$ contains one or two more $1$-cycles, respectively.

Similarly, the constructed bijection implies that the number of gray-red cycles in $G$ is the same as in $G'$, except when $G_3$ is empty in which case $G$ has one more gray-red cycle. It follows that
$$
2 \frac{\partial \G_1}{\partial x} = \big( s_1 + \sum_{i\geq 1} i s_{i+1} \frac{\partial \G_1}{\partial s_i}\big)\cdot \big( s_1 + \sum_{i\geq 1} i s_{i+1} \frac{\partial \G_1}{\partial s_i}\big)\cdot \big(u + 2x\frac{\partial \G_1}{\partial x}\big),
$$
where the expressions in the parentheses correspond to the graphs $G_1,G_2,G_3$ respectively and the operator $i s_{i+1} \frac{\partial}{\partial s_i}$ produces the generating function 
for special breakpoint graphs with the root belonging to a black-gray $i$-cycle
(by Lemma~\ref{lem:gfi})
 and replacing one $s_i$ with $s_{i+1}$. The terms $s_1$ correspond to the cases of empty $G_1$ and/or $G_2$, and the term $u$ corresponds to the case of empty $G_3$.

The PDE determines the coefficients recursively in the $x$-degree: once the coefficients of $x^k$ for $k<n$ are known, the coefficient of $x^n$ is determined.  Together with the empty-graph value
\[
\G_1(0,u,s_1,s_2,\dots)=u,
\]
this proves the stated characterization.
\end{proof}

\subsection{Shift flow and a universal one-parameter reduction}\label{sec:shift_flow}

The equation in Theorem~\ref{th:G1_PDE} involves infinitely many cycle variables, but only through a single first-order shift derivation.  This makes every target weight sequence lie on a canonical one-parameter orbit.  For later use in both strata, let
\[
C_m:=\frac{1}{m+1}\binom{2m}{m},\qquad
\C(t):=\sum_{m\ge0} C_m t^m=\frac{1-\sqrt{1-4t}}{2t},
\qquad \C(t)=1+t\C(t)^2,
\]
for the Catalan generating function.

Write
\[
L:=\sum_{i\geq1} i s_{i+1}\frac{\partial}{\partial s_i},
\qquad
E:=\frac{\partial}{\partial s_1}.
\]
Then Theorem~\ref{th:G1_PDE} can be written as
\begin{equation}\label{eq:G1_vectorfield}
2\G_{1,x}=\bigl(s_1+L\G_1\bigr)^2\bigl(u+2x\G_{1,x}\bigr).
\end{equation}

The operator $L$ also has a direct cycle-index interpretation.  For a cycle-profile monomial
\[
M_{\mathbf c}(\mathbf s):=\prod_{j\geq1}s_j^{c_j},
\]
one has
\begin{equation}\label{eq:L_cycle_lengthening}
LM_{\mathbf c}
=
\sum_{i\geq1} i c_i\,s_{i+1}s_i^{c_i-1}
\prod_{j\ne i}s_j^{c_j}.
\end{equation}
Thus $i s_{i+1}\partial_{s_i}$ marks one of the $i$-cycles together with one of its $i$ positions and then replaces its cycle weight $s_i$ by $s_{i+1}$.  In this sense $L$ is the cycle-lengthening analogue of the rooted cycle-selection operator $i s_i\partial_{s_i}$ from Lemma~\ref{lem:gfi}.  Along a shift-flow curve, differentiating a monomial performs exactly the same operation term by term; this is the combinatorial content of the chain-rule identity~\eqref{eq:shift_flow_chain_rule} below.

The shift field can be linearized directly on the generating series of the cycle weights.  Put
\begin{equation}\label{eq:weight_series_SR}
S(z):=\sum_{i\geq1}s_i z^i,
\qquad
R(z):=\frac{S(z)}{z}.
\end{equation}
Since $L$ acts only on the coefficients $s_i$,
\begin{equation}\label{eq:L_linearization}
LS
=\sum_{i\geq1}i s_{i+1}z^i
=S_z-\frac{S}{z},
\qquad
LR=R_z.
\end{equation}
Thus, if $\mathbf s(t)$ is an integral curve of $L$ and $S(z,t)$ and $R(z,t)$ are its weight series, then
\begin{equation}\label{eq:transport_shift_flow}
R_t=R_z.
\end{equation}
The infinite system defining the shift flow is therefore conjugate to the ordinary translation field $\partial_z$.  If $R(z,0)=f(z)$, its formal solution is
\begin{equation}\label{eq:shift_flow_translation}
R(z,t)=f(z+t),
\qquad
S(z,t)=z f(z+t).
\end{equation}
Equivalently, if $S_0(z):=S(z,0)=zf(z)$, then
\[
S(z,t)=\frac{z}{z+t}S_0(z+t).
\]
Coefficient extraction in~\eqref{eq:shift_flow_translation} gives the integral curves used below.

\begin{prop}[Shift-flow reduction]\label{prop:G1_shift_flow}
Fix a sequence $\mathbf w:=(w_1,w_2,\ldots)$ and define
\[
f(t):=\sum_{k\geq0}w_{k+1}t^k,
\qquad
s_i(t):=\coef{z^i}\,z f(z+t)
       =\frac{f^{(i-1)}(t)}{(i-1)!}\quad (i\geq1).
\]
Then $s_i(0)=w_i$ and
\[
\frac{d}{dt}s_i(t)=i s_{i+1}(t).
\]
Consequently, if
\[
F(x,t):=\G_1\bigl(x,u,s_1(t),s_2(t),\ldots\bigr),
\qquad
H(x,t):=F(x,t)+\int_0^t f(\tau)\,d\tau,
\]
then
\begin{equation}\label{eq:G1_reduced_generic}
2H_x=H_t^2\bigl(u+2xH_x\bigr),
\qquad
H(0,t)=u+\int_0^t f(\tau)\,d\tau.
\end{equation}
In particular, the specialization of $\G_1$ at $s_i=w_i$ is recovered at $t=0$.
\end{prop}

\begin{proof}
Equation~\eqref{eq:shift_flow_translation} gives
$S_t=S_z-S/z$, and comparison of the coefficient of $z^i$ yields
$s_i'(t)=i s_{i+1}(t)$.  Hence, by the chain rule,
\[
F_t=\sum_{i\geq1} i s_{i+1}\frac{\partial \G_1}{\partial s_i}=L\G_1
\]
along the curve.  Since $s_1(t)=f(t)$, we have $H_t=f(t)+F_t=s_1+L\G_1$, while $H_x=F_x$. Substitution into~\eqref{eq:G1_vectorfield} proves~\eqref{eq:G1_reduced_generic}.  The boundary value follows from $\G_1(0,u,s_1,s_2,\ldots)=u$.
\end{proof}

\paragraph{Ore-algebra viewpoint.}
The same linearization has a simple operator interpretation.  The natural setting is the skew-polynomial framework introduced by Ore~\cite{Ore1933}; mixed differential--shift annihilators and their algorithmic manipulation are standard in modern Ore-algebra treatments of holonomic functions~\cite{ChyzakSalvy1998}.  Let $\mathsf T_i$ denote the forward shift in the discrete index, $(\mathsf T_i a)_i:=a_{i+1}$, and let $D_t,D_z$ denote differentiation.  The coefficient system is annihilated by the differential--shift Ore operator
\[
D_t-i\mathsf T_i,
\qquad
\mathsf T_i\,i=(i+1)\mathsf T_i.
\]
After passing to the ordinary generating series, this becomes
\begin{equation}\label{eq:ore_shift_operator}
\left(D_t-D_z+z^{-1}\right)S=0.
\end{equation}
The gauge transformation $S=zR$ conjugates~\eqref{eq:ore_shift_operator} to the constant-coefficient transport operator
\[
(D_t-D_z)R=0,
\]
which is exactly~\eqref{eq:transport_shift_flow}.  Thus the ``special substitution'' for $\kappa=1$ is, in Ore-algebra language, the characteristic flow of a first-order differential--shift system.  If the initial series $f$ is D-finite in the standard sense~\cite{Stanley1980}, the translation formula places the entire shifted family inside the usual holonomic/Ore-algebra framework; the Catalan case below is algebraic, hence D-finite.  We use only the elementary translation formula~\eqref{eq:shift_flow_translation} here.

The curve itself is a property of weight space, not of the $\kappa=1$ equation.  For any formal series $\Phi(s_1,s_2,\ldots)$, the same chain-rule calculation gives
\begin{equation}\label{eq:shift_flow_chain_rule}
\frac{d}{dt}\Phi(s_1(t),s_2(t),\ldots)=L\Phi(s_1(t),s_2(t),\ldots).
\end{equation}
A two-dimensional version, used below for the diagonal of $\G_0$, is obtained by freeing the first coordinate.  For a formal series $f$, define
\begin{equation}\label{eq:decoupled_shift_surface}
\iota_f(s,t):=\left(s,f'(t),\frac{f''(t)}{2!},\frac{f^{(3)}(t)}{3!},\ldots\right).
\end{equation}
Equivalently, the weight series of this surface is
\begin{equation}\label{eq:decoupled_shift_series}
S_f(z;s,t):=z\bigl(s+f(z+t)-f(t)\bigr).
\end{equation}
If $\Phi_f:=\Phi\circ\iota_f$, then
\begin{equation}\label{eq:shift_surface_pullback}
(E\Phi)\circ\iota_f=(\Phi_f)_s,
\qquad
(L\Phi)\circ\iota_f=f'(t)(\Phi_f)_s+(\Phi_f)_t.
\end{equation}
Indeed, the second identity is just~\eqref{eq:shift_flow_chain_rule} with the $i=1$ term separated.  Thus any equation written in terms of $E$ and $L$ restricts automatically to a two-variable equation on the surface~\eqref{eq:decoupled_shift_surface}.

Thus a one-dimensional reduction exists for \emph{every} target weight sequence in the $\kappa=1$ problem.  What makes a particular specialization tractable is not the existence of the reduction, but the form of the one-variable function $f$.  For the Catalan weights $w_i=C_{i-1}$ we have $f(t)=\C(t)$, and the entire shifted sequence is encoded by
\begin{equation}\label{eq:Catalan_shift_series}
S(z,t)=z\C(z+t).
\end{equation}
Since $\C(w)=1+w\C(w)^2$, this series satisfies the quadratic identity
\begin{equation}\label{eq:Catalan_shift_algebraic}
(z+t)S(z,t)^2-zS(z,t)+z^2=0.
\end{equation}
The corresponding decoupled surface has the compact form
\[
S_{\C}(z;s,t)=z\bigl(s+\C(z+t)-\C(t)\bigr).
\]

\begin{remark}[A polygon-gluing reading of the Catalan weights]
The specialization $s_i=C_{i-1}$ also admits a topological interpretation.  The number $C_k$ counts side gluings of a $2k$-gon that produce a sphere~\cite[Lemma~13]{Alexeev2016top}.  Thus assigning the weight $C_{i-1}$ to a black-gray $i$-cycle can be viewed as decorating that cycle by a spherical gluing of a $2(i-1)$-gon.  We do not use this decoration in any proof; it is simply another language for the same multiplicative Catalan weighting.
\end{remark}

The general reduction~\eqref{eq:G1_reduced_generic} is independent of any application; at the signed value $u=-1$ the Catalan curve can in fact be solved explicitly, as follows.

\begin{prop}[Signed Catalan evaluation for $\kappa=1$]\label{prop:G1_signed_Catalan}
At $u=-1$ and $s_i=C_{i-1}$,
\begin{equation}\label{eq:G1_signed_Catalan_closed}
\G_1(x,-1,C_0,C_1,C_2,\ldots)
=
-1-\frac12\int_0^x \C(-y)\,dy.
\end{equation}
Equivalently, for every $n\geq1$,
\begin{equation}\label{eq:G1_signed_Catalan_coeff}
2n\,\coef{x^n}\,\G_1(x,-1,C_0,C_1,C_2,\ldots)
=(-1)^n C_{n-1}.
\end{equation}
\end{prop}

\begin{proof}
Apply Proposition~\ref{prop:G1_shift_flow} with $u=-1$ and $f=\C$. Write $p:=H_t$. Equation~\eqref{eq:G1_reduced_generic} becomes
\[
H_x=-\frac{p^2}{2(1-xp^2)}.
\]
Differentiating with respect to $t$ gives the quasilinear equation
\[
p_x+\frac{p}{(1-xp^2)^2}p_t=0.
\]
Hence $p$ is constant along characteristics, and a characteristic leaving the line $x=0$ at $(0,t_0)$ satisfies
\begin{equation}\label{eq:G1_Catalan_characteristic}
t=t_0+\frac{xp}{1-xp^2},
\qquad
p=\C(t_0).
\end{equation}
For the characteristic reaching $t=0$, the Catalan relation $t_0=(p-1)/p^2$ and~\eqref{eq:G1_Catalan_characteristic} give
\[
p=1-xp^2.
\]
The unique formal solution with $p(0)=1$ is therefore $p=\C(-x)$. On $t=0$ one also has $1-xp^2=p$, so
\[
\frac{d}{dx}H(x,0)=H_x(x,0)=-\frac{p}{2}=-\frac12\C(-x).
\]
Since $H(0,0)=-1$ and $H(x,0)=\G_1(x,-1,C_0,C_1,\ldots)$, integration proves~\eqref{eq:G1_signed_Catalan_closed}. Finally, $\C(-x)=\sum_{m\geq0}(-1)^mC_mx^m$, and coefficient extraction yields~\eqref{eq:G1_signed_Catalan_coeff}.
\end{proof}

The evaluation above is intrinsic to the special-breakpoint-graph generating function. It complements the signed $\kappa=0$ Catalan evaluation in Proposition~\ref{prop:Catalan_surface_closed_form}. The two formulas will be used together only later, in Section~\ref{sec:physics_application}, where the application-specific sign convention is introduced.

\section{The projective-plane class \texorpdfstring{$\kappa=0$}{kappa=0}}\label{sec:kappa=0}
The next stratum is topologically different: its canonical surface is the projective plane. We begin with the corresponding M\"obius-ladder/root-removal decomposition and its generating-function equation, then expose a finite-dimensional $\mathfrak{sl}_2$ action behind that equation and use it to derive Catalan reductions and refined coefficient formulas.

\subsection{M\"obius ladders, root removal, and the recursive equation}\label{sec:kappa0_structure}

The case of $\kappa=0$ corresponds to the special breakpoint graphs of order $4n$ with a total number of black-gray and gray-red cycles equal to $2n$. The generating function in this case starts with
\[
\begin{split}
\G_0(x,u,s_1,s_2,\dots) & = s_2 u \frac{x}2 \\
& + \big(4 s_1 s_3 u^2 + 4 s_1^2 s_2 u\big) \frac{x^2}4 \\
& + \big((s_2^3 + 6 s_1 s_2 s_3 + 9 s_1^2 s_4) u^3 + (15 s_1^2 s_2^2 + 18 s_1^3 s_3) u^2 + 9 s_1^4 s_2 u\big) \frac{x^3}6 \\
& + \dots.
\end{split}
\]
For example, in the coefficient of $\frac{x^2}4$ the terms $4 s_1 s_3 u^2$ and $4 s_1^2 s_2 u$ enumerate the permutations 
$$
\{ (3,1,2,4), (4,2,1,3), (1,3,4,2), (2,4,3,1)\}\quad\text{and}\quad \{ (2,1,3,4), (1,2,4,3), (1,3,2,4), (4,2,3,1) \},
$$ 
respectively (Fig.~\ref{fig:example}).

By Proposition~\ref{prop:canonical_surface}, every graph in the $\kappa=0$ stratum is canonically embedded in $\mathbb{RP}^2$.  The unique black-red cycle bounds one disk face, so removing the interior of that face leaves a M\"obius band.  Thus the disk that carries the noncrossing chord picture for $\kappa=1$ is replaced, at $\kappa=0$, by its one-crosscap analogue.  The exceptional family in the next lemma is therefore naturally expected to be a twisted, rather than cylindrical, ladder; the local argument below makes this precise.

\begin{lemma}\label{lem:ladder}
For a positive integer $n$, let $G$ be a special breakpoint graph of order $4n$ with the total number of black-gray and gray-red cycles equal to $2n$. Then either $G$ contains a black-gray $1$-cycle, or $G$ is a M\"obius ladder graph $M_{4n}$ with odd $n$.
\end{lemma}

\begin{proof} Note that each of $2n$ gray edges is present in one black-gray and one gray-red cycle. Furthermore, each gray-red cycle contains an even number of gray edges, and thus has at least 2 gray edges. It follows that if there are no black-gray 1-cycles, then all black-gray and gray-red cycles are exactly 2-cycles. We will show that in this case $G$ represents a M\"obius ladder $M_{4n}$ with odd $n$.

Let $\{a,b\}$ and $\{c,d\}$ be the black edges of the black-gray cycle $C$ in $G$ that contains the root $a$. Without loss of generality, we assume that along the black-red cycle in $G$ these edges are directed as $(a,b)$ and $(c,d)$. 
Let graph $G'$ be obtained from $G$ by replacing black edges $\{a,b\}$ and $\{c,d\}$ with black edges parallel to the gray edges in $C$.

If the gray edges in $C$ are $\{a,d\}$ and $\{b,c\}$, then $G'$ represents the union of two special breakpoint graphs, each of which has a single black-gray 1-cycle and the other black-gray cycles are 2-cycles. It follows however that each of these graphs contains an odd number of black edges, which is impossible. 

Hence, the gray edges in $C$ are $\{a,c\}$ and $\{b,d\}$, and thus $G'$ represents a special breakpoint graph of order $4n$ with the total number of black-gray and gray-red cycles equal to $2n+1$. Clearly, these cycles are two black-gray 1-cycles (formed by double edges $\{a,c\}$ and $\{b,d\}$), $n-1$ black-gray 2-cycles, and $n$ gray-red 2-cycles. By Lemma~\ref{lem:noncrossing}, the graph $G'$ is noncrossing. Removing from it the black edges $\{a,c\}$ and $\{b,d\}$ results in a ladder graph, where the ``sides'' are formed by black-red paths, while the ``crossbars'' are formed by the gray edges. Since each side contains $n$ red edges and starts and ends with red edges of the same oddness (imposed by $G$), it follows that $n$ is odd. Then restoring the graph $G$ by adding back black edges $\{a,b\}$ and $\{c,d\}$ essentially twists the ladder into a M\"obius ladder graph $M_{4n}$. 
\end{proof}

\begin{thm}\label{th:G0_PDE}
The generating function $\G_0(x,u,s_1,s_2,\dots)$ satisfies
\begin{equation}\label{eqn:pde_kappa_0}
\frac{\partial \G_0}{\partial s_1}
=
x\sum_{i\geq1}i\left(
 u s_{i+1}+\sum_{j=1}^i s_j s_{i+1-j}
\right)\frac{\partial \G_0}{\partial s_i},
\end{equation}
with boundary value
\begin{equation}\label{g0bound}
\G_0(x,u,0,s_2,s_3,\ldots)
=
\frac14\log\frac{1+us_2x}{1-us_2x}.
\end{equation}
\end{thm}

\begin{proof} By Lemma~\ref{lem:ladder}, the special breakpoint graphs enumerated by $\G_0$ are either M\"obius ladders $M_{4n}$ with odd $n$, or contain a black-gray 1-cycle. 

The M\"obius ladder $M_{4n}$ with odd $n$ contains $n$ black-gray cycles and $n$ gray-red cycles, each of which is a 2-cycle. Furthermore, by symmetry, such a rooted graph is unique for each $n$. It follows that such graphs contribute to $\G_0$ the term:
$$\sum_{\substack{n\geq1\\ n\text{ odd}}} \frac{x^n}{2n} u^n s_2^n = 
\frac14\log \frac{1+us_2x}{1-us_2x}.$$

Let $S_n$ be the set of special breakpoint graphs of order $4n$ with the total number of black-gray and gray-red cycles equal to $2n$, and let $S'_n\subset S_n$ be the subset of those graphs where the root belongs to a black-gray $1$ cycle. We will define a transformation of the graphs in $S'_n$ into graphs in $S_{n-1}$ as follows. 

Let $G$ be a graph from $S'_n$ rooted at a vertex $a$, and let $u$ be the other vertex in the black-gray $1$-cycle containing $a$. Let $(u,v,w)$ and $(a,b,c)$ be the adjacent red-black paths, and $\{b,s\}$ and $\{v,t\}$ be gray edges. 
We construct a graph $G'$ from $G$ by removing the vertices $v,u,a,b$ along with their incident edges, adding a black edge $\{c,w\}$, and marking the vertex $c$. In the case $\{s,t\}\ne \{b,v\}$ we also add a gray edge $\{s,t\}$ (Fig.~\ref{fig:k0transform}).
Below we will show that $G'\in S_{n-1}$, and if $\{s,t\}\ne \{b,v\}$ then the black edge $\{c,w\}$ and the gray edge $\{s,t\}$ belong to the same black-gray cycle in $G'$.

\begin{figure}[!t]
    \centering
    \includegraphics[width=.75\linewidth]{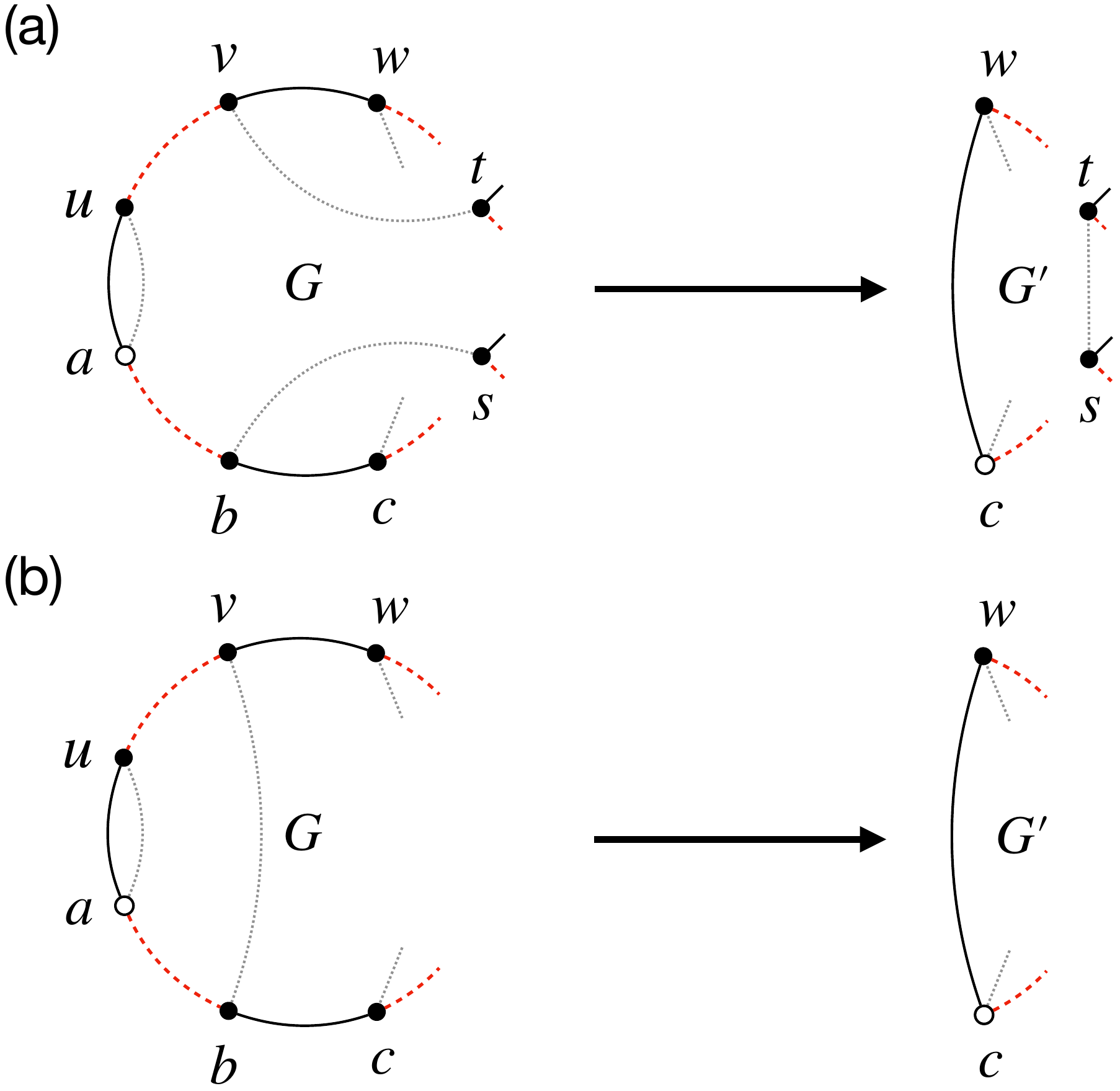}
    \caption{A transformation of a special breakpoint graph of order $4n$ rooted at vertex $a$ in a black-gray $1$-cycle into a special breakpoint graph of order $4(n-1)$. \textbf{(a)}  Case $\{s,t\}\ne \{b,v\}$. \textbf{(b)} Case $\{s,t\}=\{b,v\}$.}
    \label{fig:k0transform}
\end{figure}

\paragraph{Case I: $\{s,t\}=\{b,v\}$.} In this case, the graph $G'$ contains one fewer gray-red cycle and one fewer black-gray cycle as compared to $G$, and thus $G'\in S_{n-1}$. Furthermore, if black edge $\{w,c\}$ belongs to black-gray $i$-cycle in $G'$, then in $G$ it expands into a black-gray $(i+1)$-cycle containing black edges $\{b,c\}$ and $\{v,w\}$. 

\paragraph{Case II: $\{s,t\}\ne \{b,v\}$.} In this case, gray edges $\{b,s\}$ and $\{v,t\}$ in $G$ are distinct and belong to the same gray-red cycle, but must belong to distinct black-gray cycles, implying that $G'\in S_{n-1}$. Indeed, if they belong to the same black-gray cycle, we consider a special breakpoint graph $G''$ obtained from $G$ by replacing these edges with gray edges $\{b,v\}$ and $\{s,t\}$. The graph $G''$ has one more gray-red cycle and one more black-gray cycle (by Lemma~\ref{lem:2break}) as compared to $G$, and thus the total number of black-gray and gray-red cycles in $G''$ is $2n+2$. This would give $\kappa(G'')=2$, contradicting Proposition~\ref{prop:canonical_surface}, which gives $\kappa\leq1$. Thus, gray edges $\{b,s\}$ and $\{v,t\}$ belong to distinct black-gray cycles in $G$. By Lemma~\ref{lem:2break}, it follows that black edge $\{c,w\}$ and gray edge $\{s,t\}$ in $G'$ belong to the same black-gray $i$-cycle for some $i\geq 1$. In graph $G$, this cycle is replaced with a black-gray $j$-cycle and an $(i+1-j)$-cycle for $1\leq j\leq i$, containing gray edges $\{b,s\}$ and $\{v,t\}$, respectively. The parity argument in Lemma~\ref{lem:2break} also shows that these splitting configurations exhaust the same-cycle reconnections: the alternative ``twist'' pairing would join vertices incident with red edges of equal parity and therefore would not define a special breakpoint graph.  

Our goal is to enumerate the pairs $(G,G')$ with $G\in S'_n$ and $G'\in S_{n-1}$ constructed from $G$ as described above, along with their cycle structures, in two different ways.
On the one hand, such pairs are uniquely defined by graphs $G\in S'_n$, which by Lemma~\ref{lem:gfi} are enumerated by the coefficient of $x^n$ in
$$
s_1\frac{\partial}{\partial s_1} \G_0(x,u,s_1,s_2,\dots). 
$$
On the other hand, such pairs can be constructed from a fixed graph $G'$, allowing to enumerate the corresponding $G$ by the coefficients of $x^n$ in
$$
x s_1 \sum_{i\geq 1} i  \left( u s_{i+1} + \sum_{j=1}^i s_j s_{i+1-j}\right) \frac{\partial \G_0}{\partial s_i}
$$
where 
\begin{itemize}
\item
$i$ stands for the number of black edges in the black-gray cycle containing the root of $G'$;
\item $i\frac{\partial \G_0}{\partial s_i}$ produces the factor for the number of ways to select such a cycle (by Lemma~\ref{lem:gfi}) and also decreases the power of $s_i$ by 1;
\item the term $us_{i+1}$ accounts for the case I (Fig.~\ref{fig:k0transform}b);
\item the sum $\sum_{j=1}^i s_j s_{i+1-j}$ corresponds to case II (Fig.~\ref{fig:k0transform}a) and accounts for the ways to pick a gray edge in the same black-gray cycle and split/expand this cycle into a black-gray $j$-cycle and an $(i+1-j)$-cycle in $G$.
\end{itemize}

It follows that equation~\eqref{eqn:pde_kappa_0} holds. Setting $s_1=0$ leaves only graphs with no black-gray $1$-cycle, so the M\"obius-ladder contribution computed at the beginning of the proof is exactly the boundary value~\eqref{g0bound}. Conversely, equation~\eqref{eqn:pde_kappa_0} determines the coefficients recursively in the $x$-degree once this boundary value is fixed, so the PDE and boundary condition characterize $\G_0$ as a formal power series.
\end{proof}

\begin{corol}[Recursive integral form]\label{cor:G0_integral_form}
Equivalently, the $\kappa=0$ generating function satisfies
\begin{equation}\label{eq:G0_integral_recursive}
\G_0(x,u,s_1,s_2,\ldots)
=
\frac14\log\frac{1+us_2x}{1-us_2x}
+x\int {\rm d}s_1\,
\sum_{i\geq1}i\left(
 u s_{i+1}+\sum_{j=1}^i s_j s_{i+1-j}
\right)\frac{\partial\G_0}{\partial s_i}.
\end{equation}
Here $\int {\rm d}s_1$ denotes the formal antiderivative in $s_1$ with zero $s_1$-constant term.  The first term is the contribution of the exceptional odd M\"obius ladders, while the antiderivative term is the contribution of graphs containing a rooted black-gray $1$-cycle.
\end{corol}

\begin{proof}
Apply the zero-constant formal antiderivative $\int {\rm d}s_1$ to~\eqref{eqn:pde_kappa_0} and then use the boundary value~\eqref{g0bound}.  The combinatorial interpretation of the two terms is exactly the dichotomy in Lemma~\ref{lem:ladder} and the root-removal construction in the proof of Theorem~\ref{th:G0_PDE}.
\end{proof}

\subsection{An \texorpdfstring{$\mathfrak{sl}_2$}{sl2} structure}

The $\kappa=0$ equation is governed by several first-order derivations rather than the single shift direction of the spherical class.  Their Lie closure is nevertheless finite-dimensional, which turns the previously special-looking substitutions into canonical low-dimensional orbits.

The differential form of Theorem~\ref{th:G0_PDE} is
\begin{equation}\label{eq:G0_EV}
E\G_0=xV\G_0,
\end{equation}
where, in addition to the derivation $E:=\partial/\partial s_1$ already used in Subsection~\ref{sec:shift_flow}, we introduce the derivations
\begin{equation}\label{eq:E_D_V_def}
D:=\sum_{i\geq1} i s_i\frac{\partial}{\partial s_i},
\qquad
V:=\sum_{i\geq1}i\left(us_{i+1}+\sum_{j=1}^i s_j s_{i+1-j}\right)\frac{\partial}{\partial s_i}.
\end{equation}
We regard $u$ as a parameter for these derivations.

\begin{thm}[Lie closure of the $\kappa=0$ operators]\label{th:sl2}
The derivations $E,D,V$ satisfy
\begin{equation}\label{eq:sl2_commutators}
[E,V]=2D,
\qquad
[D,E]=-E,
\qquad
[D,V]=V.
\end{equation}
Equivalently, $V,-E,2D$ satisfy the standard $\mathfrak{sl}_2$ commutation relations.  Hence the Lie algebra generated by the two vector fields appearing in~\eqref{eq:G0_EV} is three-dimensional.
\end{thm}

\begin{proof}
Only the first commutator requires more than homogeneity.  Put
\[
A_i:=us_{i+1}+\sum_{j=1}^i s_j s_{i+1-j}.
\]
Since $\partial A_i/\partial s_1=2s_i$, including the case $i=1$, we obtain
\[
[E,V]=\sum_{i\geq1}2i s_i\frac{\partial}{\partial s_i}=2D.
\]
The relation $[D,E]=-E$ follows immediately from the weight of $s_1$.  Every monomial in $A_i$ has weighted degree $i+1$ for the grading $\deg s_j=j$, so
\[
\left[D,A_i\frac{\partial}{\partial s_i}\right]
=((i+1)-i)A_i\frac{\partial}{\partial s_i},
\]
and summing over $i$ gives $[D,V]=V$.
\end{proof}

The finite Lie closure has an immediate consequence for the infinite-variable equation.

\begin{corol}[Finite-dimensional orbit reduction]\label{cor:sl2_orbit}
Fix $u$ and a weight sequence $\mathbf w:=(w_1,w_2,\ldots)$.  Let $\mathcal O_{\mathbf w}$ denote the formal orbit through $\mathbf w$ generated by the flows of $E,D,V$.  Then $\mathcal O_{\mathbf w}$ has dimension at most three, and the restriction of $\G_0$ to $\mathcal O_{\mathbf w}$ satisfies a first-order PDE in at most three auxiliary variables.

More generally, if a nonzero element of $\operatorname{span}\{E,D,V\}$ annihilates $\mathbf w$, then the orbit through $\mathbf w$ has dimension at most two.  Thus a nontrivial stabilizer of the target weight sequence produces an automatic dimension drop in the reduction of~\eqref{eq:G0_EV}.
\end{corol}

\begin{proof}
By Theorem~\ref{th:sl2}, the Lie algebra generated by the vector fields $E$ and $V$ appearing in~\eqref{eq:G0_EV} is the three-dimensional space $\operatorname{span}\{E,D,V\}$.  Hence the formal orbit through any point has dimension at most three.  Choose formal local coordinates $\theta_1,\ldots,\theta_d$ on such an orbit, where $d\leq3$.  Since $E$ and $V$ are tangent to the orbit, their restrictions have the form
\[
E=\sum_{j=1}^d e_j(\boldsymbol\theta)\frac{\partial}{\partial\theta_j},
\qquad
V=\sum_{j=1}^d v_j(\boldsymbol\theta)\frac{\partial}{\partial\theta_j}.
\]
Pulling back~\eqref{eq:G0_EV} to the orbit therefore gives
\[
\sum_{j=1}^d e_j(\boldsymbol\theta)\frac{\partial \G_0}{\partial\theta_j}
=x\sum_{j=1}^d v_j(\boldsymbol\theta)\frac{\partial \G_0}{\partial\theta_j},
\]
a first-order PDE in $d\leq3$ auxiliary variables.  If a nonzero Lie-algebra element annihilates $\mathbf w$, then the stabilizer at $\mathbf w$ has positive dimension, so the orbit dimension is at most two.
\end{proof}

The corollary turns the search for a ``magic substitution'' into the search for a target sequence with a nontrivial stabilizer.  This stabilizer condition can itself be tested by a one-variable differential equation.  Define
\begin{equation}\label{eq:S_Q_def}
S(z):=\sum_{i\geq1}s_i z^i,
\qquad
Q(z):=\frac u2+S(z).
\end{equation}
The three derivations act on $Q$ as
\begin{equation}\label{eq:EDV_on_Q}
EQ=z,
\qquad
DQ=zQ_z,
\qquad
VQ=2QQ_z-\frac{Q^2-u^2/4}{z}.
\end{equation}
Indeed, if $A_i$ is as above, then $A_i=\coef{z^{i+1}}(uS+S^2)$, so
\[
VQ=\sum_{i\geq1} iA_i z^i
=z\frac{d}{dz}\left(\frac{uS+S^2}{z}\right)
=2QQ_z-\frac{Q^2-u^2/4}{z}.
\]
It follows that a sequence is annihilated by a nonzero combination $aE+bD+cV$ precisely when its generating series satisfies
\begin{equation}\label{eq:stabilizer_ODE}
az+bzQ_z+c\left(2QQ_z-\frac{Q^2-u^2/4}{z}\right)=0.
\end{equation}
Thus the search for two-dimensional invariant specializations of the infinite sequence space can be reduced to a first-order ODE for a single generating series.

We now apply the orbit-reduction principle of Corollary~\ref{cor:sl2_orbit}.  For the full $\kappa=0$ equation at $u=-1$, the Catalan weight sequence turns out to be annihilated by $V$.  The corollary then guarantees that its $E,D,V$ orbit is at most two-dimensional; the calculation below supplies explicit coordinates on this orbit and the restricted PDE.  After taking the diagonal $u^n x^n$, a different vector field survives, namely the shift field $L$, so the same Catalan point is instead followed along its canonical derivative flow.  We treat these two reductions in turn.

\subsection{Catalan orbit reduction and the signed fixed point}\label{sec:catalan_specializations}

The vector field $V$ itself has a simple characteristic description in terms of~\eqref{eq:S_Q_def}.  If $Q(z,t)$ follows the $V$-flow, then
\begin{equation}\label{eq:V_flow_Q}
Q_t=2QQ_z-\frac{Q^2-u^2/4}{z}.
\end{equation}
Set
\[
y:=\frac{Q^2-u^2/4}{z}.
\]
Along a characteristic of~\eqref{eq:V_flow_Q},
\[
\dot z=-2Q,
\qquad
\dot Q=-y,
\qquad
\dot y=0.
\]
Hence the $V$-flow is integrable by characteristics: locally one may write
\begin{equation}\label{eq:V_characteristic_solution}
Q+t y=F(y),
\qquad
y=\frac{Q^2-u^2/4}{z},
\end{equation}
where $F$ is fixed by the initial sequence.

For the Catalan sequence $s_i=C_{i-1}$, so that $S(z)=z\C(z)$ and $\C(z)=1+z\C(z)^2$, the initial characteristic invariant is
\begin{equation}\label{eq:Catalan_y_generic_u}
y_0(z)=(u+1)\C(z)-1.
\end{equation}
For $u\neq-1$ this can be inverted explicitly, and~\eqref{eq:V_characteristic_solution} is determined by
\begin{equation}\label{eq:Catalan_F_generic_u}
F(y)=\frac u2+\frac{y-u}{y+1}.
\end{equation}
The same special value can be seen directly from the interaction between $V$ and the shift field $L$.  Consider, for arbitrary $u$, the two-parameter Catalan surface
\begin{equation}\label{eq:Catalan_surface_generic_u}
s_1:=s,
\qquad
s_i:=C_{i-1}q^i\quad(i\geq2).
\end{equation}
Here $\partial_q$ denotes differentiation along this parameterized surface, while $L$ is the ambient shift derivation from Subsection~\ref{sec:shift_flow}.

\begin{prop}[Catalan transversality]\label{prop:Catalan_transversality}
Along the surface~\eqref{eq:Catalan_surface_generic_u}, the ambient vector field $V$ decomposes as
\begin{equation}\label{eq:Catalan_transversality}
V
=(s^2-q^2)E+2q(s-q)\partial_q+(u+1)L.
\end{equation}
Thus the failure of the Catalan surface to be $V$-invariant is carried precisely by the shift direction $L$, with coefficient $u+1$.  In particular, at $u=-1$ this transverse term disappears and the surface is $V$-invariant.
\end{prop}

\begin{proof}
The $s_1$ component is
\[
Vs_1=us_2+s_1^2=s^2+uq^2.
\]
For $i\geq2$, the Catalan convolution gives
\begin{align*}
\frac1iVs_i
&=uC_iq^{i+1}+2sC_{i-1}q^i
 +q^{i+1}\sum_{a=1}^{i-2}C_aC_{i-1-a}\\
&=(u+1)C_iq^{i+1}+2(s-q)C_{i-1}q^i\\
&=(u+1)s_{i+1}+2(s-q)s_i.
\end{align*}
Since $Ls_1=s_2=q^2$, $Ls_i=is_{i+1}$, and $q\partial_qs_i=is_i$ for $i\geq2$, these component identities are exactly~\eqref{eq:Catalan_transversality}.  Equivalently, the underlying Catalan quadratic identity
\[
T_q(z)^2-T_q(z)+qz=0,
\qquad
T_q(z):=qz\C(qz),
\]
may be written as
\[
\left(\frac u2+T_q(z)\right)^2
=\frac{u^2}{4}-qz+(u+1)T_q(z),
\]
which makes the factor $u+1$ explicit.
\end{proof}

The value $u=-1$ is therefore singular in two equivalent senses:~\eqref{eq:Catalan_y_generic_u} becomes $y_0\equiv-1$, and Proposition~\ref{prop:Catalan_transversality} shows that the transverse shift component of $V$ vanishes.  Corollary~\ref{cor:sl2_orbit} then gives the dimension drop automatically; it remains only to parameterize the resulting orbit and write the restricted vector fields explicitly.

\begin{corol}[Catalan fixed point and a two-dimensional invariant surface]\label{cor:Catalan_orbit}
At $u=-1$, let
\[
s_i^*:=C_{i-1}\qquad (i\geq1).
\]
Then $V$ annihilates the sequence $\mathbf s^*$.
Equivalently, with
\[
Q_*(z):=-\frac12+z\C(z),
\]
one has
\begin{equation}\label{eq:Catalan_square_root}
Q_*(z)^2=\frac14-z
\qquad\text{and}\qquad
VQ_*=0.
\end{equation}
The formal surface obtained from this point by the $D$- and $E$-flows is the specialization of~\eqref{eq:Catalan_surface_generic_u} at $u=-1$; we denote it by $\Sigma_{\C}$.  The Catalan point is $(s,q)=(1,1)$.  On this surface,
\begin{equation}\label{eq:EDV_Catalan_surface}
E=\partial_s,
\qquad
D=s\partial_s+q\partial_q,
\qquad
V=(s^2-q^2)\partial_s+2q(s-q)\partial_q.
\end{equation}
Consequently, if
\[
\mathcal F(x,s,q):=\G_0\bigl(x,-1,s,C_1q^2,C_2q^3,C_3q^4,\ldots\bigr),
\]
then~\eqref{eq:G0_EV} reduces to the two-variable first-order PDE
\begin{equation}\label{eq:G0_Catalan_reduced_q}
\bigl(1-x(s^2-q^2)\bigr)\mathcal F_s
=2xq(s-q)\mathcal F_q.
\end{equation}
Equivalently, writing $q=e^t$,
\begin{equation}\label{eq:G0_Catalan_reduced_t}
\bigl(1-x(s^2-e^{2t})\bigr)\mathcal F_s
=2x(s-e^t)\mathcal F_t.
\end{equation}
\end{corol}

\begin{proof}
The Catalan identity gives
\[
\left(-\frac12+z\C(z)\right)^2
=\frac14-z\C(z)+z^2\C(z)^2
=\frac14-z,
\]
which proves~\eqref{eq:Catalan_square_root}.  Equation~\eqref{eq:EDV_on_Q} then gives $VQ_*=0$.

The $D$-flow scales $s_i$ by $q^i$, and the $E$-flow translates only $s_1$, so their orbit through the Catalan point is precisely $\Sigma_{\C}$.  The restrictions of $E$ and $D$ are immediate.  Setting $u=-1$ in Proposition~\ref{prop:Catalan_transversality} removes the $L$ term and gives the displayed restriction of $V$.  Substitution into~\eqref{eq:G0_EV} yields~\eqref{eq:G0_Catalan_reduced_q}, and $q=e^t$ gives~\eqref{eq:G0_Catalan_reduced_t}.
\end{proof}

The preceding coordinates record the cycle weights directly, but the geometry of the orbit becomes simpler after one more change of variables.

\begin{prop}[Projective coordinates on the Catalan orbit]\label{prop:Catalan_projective_coordinates}
On $\Sigma_{\C}$, set
\begin{equation}\label{eq:Catalan_alpha_def}
\alpha:=s-q,
\qquad
\widehat{\mathcal F}(x,\alpha,q):=\mathcal F(x,\alpha+q,q).
\end{equation}
Then the full weight generating series on the surface is
\begin{equation}\label{eq:Catalan_projective_Q}
Q_{\alpha,q}(z)
:=-\frac12+\sum_{i\geq1}s_i z^i
=\alpha z-\frac12\sqrt{1-4qz}.
\end{equation}
In the coordinates $(\alpha,q)$, the three generators restrict to
\begin{equation}\label{eq:EDV_Catalan_projective}
E=\partial_\alpha,
\qquad
D=\alpha\partial_\alpha+q\partial_q,
\qquad
V=\alpha^2\partial_\alpha+2\alpha q\partial_q,
\end{equation}
and the reduced equation~\eqref{eq:G0_Catalan_reduced_q} becomes
\begin{equation}\label{eq:G0_Catalan_projective_PDE}
(1-x\alpha^2)\widehat{\mathcal F}_\alpha
=2x\alpha q\widehat{\mathcal F}_q.
\end{equation}
Moreover, the flows of $E,D,V$ are
\begin{equation}\label{eq:Catalan_projective_flows}
\begin{aligned}
\exp(tE):(\alpha,q)&\longmapsto(\alpha+t,q),\\
\exp(tD):(\alpha,q)&\longmapsto(e^t\alpha,e^tq),\\
\exp(tV):(\alpha,q)&\longmapsto
\left(\frac{\alpha}{1-t\alpha},\frac{q}{(1-t\alpha)^2}\right).
\end{aligned}
\end{equation}
Thus the action on $\alpha$ is the standard projective action of $\mathfrak{sl}_2$, while $q$ transforms by the derivative of the corresponding M\"obius map.  More generally, a time-dependent vector field
$a(t)E+b(t)D+c(t)V$ induces
\begin{equation}\label{eq:Catalan_Riccati_system}
\dot\alpha=a+b\alpha+c\alpha^2,
\qquad
\dot q=(b+2c\alpha)q,
\end{equation}
so the infinite weight evolution on this orbit reduces to a Riccati equation together with its linear variational equation.
\end{prop}

\begin{proof}
On $\Sigma_{\C}$,
\[
\sum_{i\geq1}s_i z^i
=sz+\sum_{i\geq2}C_{i-1}q^iz^i
=(s-q)z+qz\C(qz).
\]
Using $qz\C(qz)=(1-\sqrt{1-4qz})/2$ and $\alpha=s-q$ gives~\eqref{eq:Catalan_projective_Q}.  Under the change of coordinates $s=\alpha+q$ one has
\[
\partial_s=\partial_\alpha,
\qquad
\left.\partial_q\right|_s
=\left.\partial_q\right|_\alpha-\partial_\alpha.
\]
Substitution into~\eqref{eq:EDV_Catalan_surface} yields~\eqref{eq:EDV_Catalan_projective}; applying the same change to~\eqref{eq:G0_Catalan_reduced_q} gives~\eqref{eq:G0_Catalan_projective_PDE}.  Integrating the three vector fields gives~\eqref{eq:Catalan_projective_flows}.  Each flow has the form $\alpha\mapsto\phi(\alpha)$ and $q\mapsto q\phi'(\alpha)$, which is the tangent lift of the projective action.  Finally, taking the indicated linear combination in~\eqref{eq:EDV_Catalan_projective} gives~\eqref{eq:Catalan_Riccati_system}.
\end{proof}

\begin{remark}[Relation with Lie systems]
The vector fields $\partial_\alpha$, $\alpha\partial_\alpha$, and $\alpha^2\partial_\alpha$ in~\eqref{eq:EDV_Catalan_projective} are the classical projective realization of $\mathfrak{sl}_2$, and the first equation in~\eqref{eq:Catalan_Riccati_system} is the associated Riccati Lie system; see, for example,~\cite{CarinenaDeLucas2011}.  The $q$-equation is the induced tangent equation, since under a projective map $\alpha\mapsto\phi(\alpha)$ one has $q\mapsto q\phi'(\alpha)$.  What is specific here is that this familiar finite-dimensional projective dynamics appears as an invariant orbit inside the infinite space of cycle weights.
\end{remark}

The reduced equation is not only finite-dimensional but explicitly solvable.

\begin{prop}[Closed form on the signed Catalan surface]\label{prop:Catalan_surface_closed_form}
The restriction $\mathcal F$ in Corollary~\ref{cor:Catalan_orbit} is
\begin{equation}\label{eq:Catalan_surface_closed_form}
\mathcal F(x,s,q)
=-\frac18\log\!\left[
1+4x\left(\frac{q}{1-x(s-q)^2}\right)^2
\right].
\end{equation}
Equivalently, in the projective coordinates of Proposition~\ref{prop:Catalan_projective_coordinates},
\begin{equation}\label{eq:Catalan_surface_closed_form_alpha}
\widehat{\mathcal F}(x,\alpha,q)
=-\frac18\log\!\left[
1+4x\left(\frac{q}{1-x\alpha^2}\right)^2
\right].
\end{equation}
In particular, at the Catalan point,
\begin{equation}\label{eq:G0_signed_Catalan_closed}
\G_0(x,-1,C_0,C_1,C_2,\ldots)
=-\frac18\log(1+4x),
\end{equation}
and therefore
\begin{equation}\label{eq:G0_signed_Catalan_coeff}
2n\,\coef{x^n}\,\G_0(x,-1,C_0,C_1,C_2,\ldots)
=(-1)^n4^{n-1}\qquad(n\geq1).
\end{equation}
\end{prop}

\begin{proof}
For fixed $x$, use the projective form~\eqref{eq:G0_Catalan_projective_PDE}.  A characteristic system is
\[
\dot\alpha=1-x\alpha^2,
\qquad
\dot q=-2x\alpha q.
\]
Therefore
\[
\frac{d}{d\lambda}\log q=-2x\alpha
=\frac{d}{d\lambda}\log(1-x\alpha^2),
\]
so
\begin{equation}\label{eq:Catalan_characteristic_invariant}
\eta:=\frac{q}{1-x\alpha^2}
\end{equation}
is constant along characteristics.  At $s=0$, equivalently $\alpha=-q$, the boundary value inherited from Theorem~\ref{th:G0_PDE} is
\[
\mathcal F(x,0,q)
=\frac14\log\frac{1-xq^2}{1+xq^2}.
\]
On that boundary, $\eta=q/(1-xq^2)$, and therefore
\[
1+4x\eta^2
=\left(\frac{1+xq^2}{1-xq^2}\right)^2.
\]
Thus the boundary value is $-\frac18\log(1+4x\eta^2)$.  Since $\eta$ is the characteristic invariant, this gives~\eqref{eq:Catalan_surface_closed_form}.  Setting $(s,q)=(1,1)$ proves~\eqref{eq:G0_signed_Catalan_closed}, and coefficient extraction gives~\eqref{eq:G0_signed_Catalan_coeff}.
\end{proof}

Equations~\eqref{eq:Catalan_transversality},~\eqref{eq:EDV_Catalan_projective}, and~\eqref{eq:Catalan_surface_closed_form} give a structural interpretation of the ``magic substitution'' that motivated this calculation: the Catalan surface is selected because the obstruction to tangency of $V$ is exactly the term $(u+1)L$, while at $u=-1$ its intrinsic dynamics is the tangent lift of the standard projective $\mathfrak{sl}_2$ action.  In the coordinate $\alpha=s-q$, the reduced PDE has the simple characteristic invariant~\eqref{eq:Catalan_characteristic_invariant}.

Equation~\eqref{eq:Catalan_projective_Q} also clarifies the square-root ansatz: the slice $\alpha=0$ is precisely the square-root family through the Catalan point, and the $E$-direction adds the single linear deformation $\alpha z$.  More generally, the stationary equation $VQ=0$ is
\[
2QQ_z=\frac{Q^2-u^2/4}{z}.
\]
Thus $Y(z):=Q(z)^2-u^2/4$ satisfies $Y'=Y/z$, and every $V$-stationary sequence has
\begin{equation}\label{eq:V_stationary_family}
Q(z)^2=\frac{u^2}{4}+cz
\end{equation}
for some constant $c$.  The Catalan point at $u=-1$ is the member $c=-1$.  In particular, a square-root ansatz of this form should be viewed as the general fixed-point family of the vector field $V$, rather than as an ad hoc guess.

\subsection{Diagonal Catalan shift-flow and noncrossing trees}\label{sec:trees}

The Catalan point has a second, complementary role after taking a diagonal of $\G_0$.  Unlike the signed specialization above, this application is governed by the shift vector field $L$ from Subsection~\ref{sec:shift_flow}.  Thus the derivative substitution below is not an independent guess: it is the canonical Catalan shift-flow through the weight sequence $(C_0,C_1,C_2,\ldots)$.  This reduction leads to a generating function for a classical noncrossing-tree statistic.

Let $\Delta$ be the diagonal of $\G_0$ with respect to $x$ and $u$, that is
\[
\Delta(x,s_1,s_2,\dots) := \sum_{n\geq 1} \frac{x^n}{2n} \sum_{\substack{c_1+2c_2+3c_3+\cdots=2n\\ c_1+c_2+c_3+\cdots=n}} g_0(n,c_1,c_2,\dots)  s_1^{c_1} s_2^{c_2} \cdots.
\]
Extracting the diagonal $u^n x^n$ from the recursive integral form in Corollary~\ref{cor:G0_integral_form} gives
\begin{equation}\label{eqn:pde_diag_0}
\Delta = \frac14\log \frac{1+s_2x}{1-s_2x} + x \int {\rm d}s_1 \sum_{i\geq 1} i  s_{i+1} \frac{\partial \Delta}{\partial s_i},
\end{equation}
with boundary condition
\[
\Delta(x,0,s_2,\dots) = \frac14\log \frac{1+s_2x}{1-s_2x}.
\]

Apply the decoupled shift surface~\eqref{eq:decoupled_shift_surface} to $f=\C$, and define the pullback
\[
\mathcal D(x,s,t):=\Delta\!\left(x,s,\frac{\C'(t)}{1!},\frac{\C''(t)}{2!},\dots\right).
\]
Thus $\mathcal D=\Delta\circ\iota_{\C}$ in the cycle variables, and since $C_0=1$,
\[
\mathcal D(x,1,0)=\Delta(x,C_0,C_1,C_2,\ldots).
\]

\begin{lemma}[Diagonal Catalan shift-flow]\label{lem:h0_pde}
The function $\mathcal D(x,s,t)$ satisfies
\[
(1-x\C'(t))\mathcal D_s = x\mathcal D_t
\]
with boundary condition
\[
\mathcal D(x,0,t) = \frac14\log \frac{1+\C'(t)x}{1-\C'(t)x}.
\]
\end{lemma}

\begin{proof}
Differentiating~\eqref{eqn:pde_diag_0} with respect to $s_1$ gives the operator equation
\begin{equation}\label{eq:diagonal_EL}
E\Delta=xL\Delta.
\end{equation}
The definition of $\mathcal D$ is precisely the pullback of $\Delta$ to the surface $\iota_{\C}$.  Therefore~\eqref{eq:shift_surface_pullback} turns~\eqref{eq:diagonal_EL} directly into
\[
\mathcal D_s=x\bigl(\C'(t)\mathcal D_s+\mathcal D_t\bigr),
\]
which is the asserted PDE after rearrangement.  The boundary condition is the pullback of the boundary value of $\Delta$ at $s_1=0$.
\end{proof}

\begin{thm}\label{th:tree_specialization}
We have
\[
2\mathcal D_x(x,1,0) = \frac{h(x)^2}{(3h(x)^2-x)^2},
\]
where
\[
h(x) := \sum_{n\geq 1} \frac{1}{2n-1}\binom{3n-3}{n-1} x^n
\]
is characterized by
\[
h(x)=x\C(h(x)).
\]
\end{thm}

\begin{proof}
For fixed $x$, the characteristic invariant of the first-order PDE in Lemma~\ref{lem:h0_pde} is $xs+t-x\C(t)$.  Hence there exists a formal series $f$ such that
\[
\mathcal D(x,s,t)=f(x,xs+t-x\C(t)),
\]
where the boundary condition determines $f$ through
\[
f(x,t-x\C(t)) = \frac14\log \frac{1+\C'(t)x}{1-\C'(t)x}.
\]
Lagrange inversion shows that the unique formal-power-series solution of $t-x\C(t)=0$ is $t=h(x)$ with the coefficients displayed above. Consequently,
\[
\mathcal D(x,1,0)=f(x,0)=\frac14\log \frac{1+\C'(h(x))x}{1-\C'(h(x))x}.
\]
Differentiating with respect to $x$ gives
\[
2\mathcal D_x(x,1,0)=\frac{(x\C'(h(x)))'}{1-\C'(h(x))^2x^2}
=\frac{h(x)^2}{(3h(x)^2-x)^2},
\]
where the last equality follows by differentiating $h=x\C(h)$ and using the Catalan identity $\C(z)=1+z\C(z)^2$ to eliminate $\C(h)$ and $\C'(h)$.
\end{proof}

\begin{corol}\label{cor:noncrossing_trees}
The series
\[
2\Delta_x(x,C_0,C_1,C_2,\dots)
= \sum_{n\geq 1} x^{n-1}
\sum_{\substack{c_1+2c_2+3c_3+\cdots=2n\\ c_1+c_2+c_3+\cdots=n}}
 g_0(n,c_1,c_2,\dots) C_0^{c_1}C_1^{c_2}\cdots
\]
equals the generating function for the sum of the levels of all nodes over all noncrossing trees with $n$ edges.
\end{corol}

\begin{proof}
Noncrossing trees on a circle were enumerated systematically by Noy~\cite{Noy1998}, and their height and path-length statistics were studied by Deutsch and Noy~\cite{DeutschNoy2002}.  The connection here is exact at the level of their classical generating functions.  Put $g(x):=h(x)/x$.  Since $h=x\C(h)$, the Catalan identity gives
\[
g=1+xg^3,
\]
the standard ternary/noncrossing-tree equation.  Multiplying the identity in Theorem~\ref{th:tree_specialization} by $x$ and eliminating $x$ with $xg^3=g-1$ gives
\begin{equation}\label{eq:noncrossing_tree_pathlength_gf}
x\,2\mathcal D_x(x,1,0)
=\frac{g(x)(g(x)-1)}{(3-2g(x))^2}.
\end{equation}
This is the generating function for the total path length (equivalently, the sum of the levels of all nodes) over noncrossing trees obtained in~\cite{DeutschNoy2002}; the same sequence is recorded in the OEIS as A062236~\cite{OEIS-A062236}. This proves the corollary.
\end{proof}

Thus the corollary identifies the breakpoint-graph diagonal with a classical noncrossing-tree statistic, not merely with a coincident numerical sequence.  What remains open is a direct bijection explaining the node-level statistic in terms of the special-breakpoint-graph cycle structure.

We next turn from finite-dimensional orbit reductions to coefficientwise consequences of the $\kappa=0$ series.  The first result fixes the number of gray-red cycles and proves polynomiality in $n$; the second refines this polynomiality by the complete nontrivial permutation cycle type.

\subsection{Polynomiality at fixed gray-red cycle number}\label{sec:polynomiality}

The $\kappa=0$ stratum has an additional finiteness property that is not apparent from the infinite-variable PDE.  Fix constants
\[
\mathbf w:=(w_1,w_2,\ldots),\qquad w_1=1,
\]
and write
\[
B_m^{\mathbf w}(n):=\coef{x^n u^m}\,\G_0(x,u,w_1,w_2,\ldots).
\]
Thus $B_m^{\mathbf w}(n)$ is the weighted enumeration of special breakpoint graphs with exactly $m$ gray-red cycles, including the normalization $1/(2n)$ from the definition of $\G_0$.  The condition $w_1=1$ is natural for weights attached only to nontrivial permutation cycles; in particular, it holds for the Catalan specialization $w_i=C_{i-1}$.

\begin{thm}[Fixed-$u$ polynomiality]\label{th:fixed_u_poly}
For every fixed integer $m\geq1$ and every fixed weight sequence $\mathbf w$ with $w_1=1$, the quantity $B_m^{\mathbf w}(n)$ is a polynomial in $n$ for all integers $n\geq1$, of degree at most $2m-1$.

More precisely, it is a finite linear combination of binomial polynomials
\[
\binom{n+a-1}{r-1},
\qquad
m+1\le r\le 2m,\qquad 0\le a\le \frac r2.
\]
In particular, since a gray-red cycle contains at least two gray edges, $B_m^{\mathbf w}(n)=0$ for $1\le n<m$; hence this polynomial is divisible by
\[
(n-1)(n-2)\cdots(n-m+1).
\]
\end{thm}

\begin{proof}
We use the permutation representation of special breakpoint graphs.  Let $\tau\in S_{2n}$ contribute to $\coef{u^m}\G_0$.  The exponent of $u$ is the number $\xi(\tau)$ of gray-red cycles, so $\xi(\tau)=m$.  Since $\kappa=0$,
\[
\#(\tau)+\xi(\tau)=2n,
\]
where $\#(\tau)$ is the number of cycles of $\tau$.  Therefore
\begin{equation}\label{eq:transposition_length_m}
2n-\#(\tau)=m.
\end{equation}
The left-hand side is the transposition length of $\tau$.

Let $S:=\operatorname{supp}(\tau)$ be the support of $\tau$, i.e., the set of points moved by $\tau$, and put $r:=|S|$.  If the restriction of $\tau$ to $S$ has $q$ cycles, then
\[
\#(\tau)=(2n-r)+q,
\qquad\text{and hence}\qquad
m=r-q.
\]
Every cycle of this restriction has length at least two, so $q\le r/2$.  Also $q\ge1$.  It follows that
\begin{equation}\label{eq:support_bound}
m+1\le r\le2m.
\end{equation}
Thus, for fixed $m$, every contributing permutation moves a bounded number of positions, independently of $n$.

We next show that the condition $\xi(\tau)=m$ is itself local on these moved positions.  List the elements of $S$ in cyclic order around $1,2,\ldots,2n$, choose one of them as distinguished, and denote the resulting ordered list by $j_1,\ldots,j_r$.  Let
\[
\varepsilon:=(\varepsilon_1,\ldots,\varepsilon_r)\in\{0,1\}^r
\]
record their parities, and define $\sigma\in S_r$ by the relations
\[
\tau(j_i)=j_{\sigma(i)}\qquad(1\le i\le r).
\]
Since $S$ is the support, $\sigma$ has no fixed points.

For the identity permutation, the gray matching coincides with the black matching, and the gray-red union is the distinguished single alternating cycle.  Passing from the identity to $\tau$ replaces exactly the $r$ gray edges
\[
\{-j_i,+j_i\}
\quad\text{by}\quad
\{-j_i,+j_{\sigma(i)}\}.
\]
Delete the former $r$ gray edges from the identity gray-red cycle and contract each of the resulting alternating paths to a single edge.  Changing the length of any interval between consecutive support positions by two merely inserts one gray-red two-edge segment into such a path, and therefore does not change which boundary signs are connected.  Consequently the contracted path matching depends only on the cyclic parity word $\varepsilon$, while the new gray matching depends only on $\sigma$.  Hence the number of components after reconnection, namely $\xi(\tau)$, is a function
\[
\xi_{\mathrm{loc}}(\varepsilon,\sigma)
\]
of this finite local data alone; it is independent of the actual gaps between the support positions.

Let $t(\varepsilon)$ be the number of parity changes in the cyclic word,
\[
t(\varepsilon):=\#\{i:\varepsilon_i\ne\varepsilon_{i+1}\},
\qquad \varepsilon_{r+1}:=\varepsilon_1.
\]
This number is even.  We now count embeddings of one fixed pointed local type $(\varepsilon,\sigma)$ into the cyclic set $\{1,\ldots,2n\}$.  The distinguished support position has $n$ possible images of its prescribed parity.  If $d_i$ is the positive cyclic gap from $j_i$ to $j_{i+1}$, then
\[
d_i=\begin{cases}
2a_i+1,&\varepsilon_i\ne\varepsilon_{i+1},\\
2a_i+2,&\varepsilon_i=\varepsilon_{i+1},
\end{cases}
\qquad a_i\ge0.
\]
The relation $d_1+\cdots+d_r=2n$ becomes
\[
a_1+\cdots+a_r=n-r+\frac{t(\varepsilon)}2.
\]
Thus the number of gap vectors is
\begin{equation}\label{eq:local_embedding_count}
\binom{n+t(\varepsilon)/2-1}{r-1}.
\end{equation}
Here, as usual, the binomial coefficient is zero when its nonnegative upper argument is smaller than $r-1$; viewed as a function of $n$, the displayed expression is the corresponding binomial polynomial.

For fixed $m$ and $r$, let $\mathcal T_{m,r}$ be the finite set of pointed pairs $(\varepsilon,\sigma)$ for which $\sigma$ has no fixed points and
\[
r-\#(\sigma)=m,
\qquad
\xi_{\mathrm{loc}}(\varepsilon,\sigma)=m.
\]
For such a type define its cycle weight by
\[
W_{\mathbf w}(\sigma):=\prod_{\gamma\in\operatorname{cyc}(\sigma)}w_{|\gamma|}.
\]
All fixed points of the full permutation lie outside $S$ and contribute the factor $w_1=1$, so this is the complete weight.  Formula~\eqref{eq:local_embedding_count} counts pointed embeddings.  Every unpointed permutation with support size $r$ is represented exactly $r$ times, according to the choice of distinguished support point.  Taking also the $1/(2n)$ normalization in $\G_0$ into account gives the exact finite formula
\begin{equation}\label{eq:fixed_u_finite_formula}
B_m^{\mathbf w}(n)
=
\sum_{r=m+1}^{2m}\frac1{2r}
\sum_{(\varepsilon,\sigma)\in\mathcal T_{m,r}}
W_{\mathbf w}(\sigma)
\binom{n+t(\varepsilon)/2-1}{r-1}.
\end{equation}
Each summand has degree $r-1$, and~\eqref{eq:support_bound} gives $r-1\le2m-1$.  This proves polynomiality and the asserted degree bound.

Finally, every gray-red cycle contains an even positive number of gray edges, hence at least two, while the graph has $2n$ gray edges in total.  Therefore $m\le n$ whenever the coefficient is nonzero.  The polynomial $B_m^{\mathbf w}(n)$ consequently vanishes at $n=1,\ldots,m-1$, giving the stated divisibility.
\end{proof}

\paragraph{Catalan specialization.}
Taking $w_i=C_{i-1}$ gives
\[
B_m(n):=\coef{x^n u^m}\,\G_0(x,u,C_0,C_1,C_2,\ldots),
\]
and Theorem~\ref{th:fixed_u_poly} shows that $B_m(n)$ is a polynomial of degree at most $2m-1$.  The first two cases obtained from the finite local formula are
\[
B_1(n)=\frac n2,
\qquad
B_2(n)=\frac{n(n-1)(5n+2)}{12}.
\]
The second identity explains the polynomial behavior visible already in the first coefficients of $\G_0$.  More generally, formula~\eqref{eq:fixed_u_finite_formula} gives a finite algorithm for every fixed $m$: one enumerates permutations on at most $2m$ support positions together with their parity words, computes the reduced gray-red component count $\xi_{\mathrm{loc}}$, and then sums the resulting binomial polynomials. No expansion of the infinite-variable PDE is required.

\subsection{Refinement by permutation cycle type}\label{sec:cycle_type_refinement}

The fixed-$u$ polynomial $B_m(n)$ combines all nontrivial permutation cycle types of transposition length $m$.  The finite-support proof of Theorem~\ref{th:fixed_u_poly} can be sharpened before the cycle weights are specialized: each individual cycle type is polynomial in $n$ as well.

We encode the nontrivial cycle type by the partition
\[
\lambda:=1^{d_1}2^{d_2}\cdots\vdash m,
\]
where a part $j$ records a permutation cycle of length $j+1$, and $d_j$ is its multiplicity. We call $\lambda$ the \emph{reduced nontrivial cycle type}. Put
\[
q=q(\lambda):=\sum_{j\ge1}d_j,
\qquad
r=r(\lambda):=\sum_{j\ge1}(j+1)d_j=m+q.
\]
Thus a permutation of type $\lambda$ has $q$ nontrivial cycles, with $d_j$ cycles of length $j+1$, and its support has size $r$.  Define the unnormalized coefficient
\begin{equation}\label{eq:Anlambda_def}
A_{n,\lambda}
:=
2n\,\coef{x^n u^m s_1^{2n-r}\prod_{j\ge1}s_{j+1}^{d_j}}\,
\G_0(x,u,s_1,s_2,\ldots).
\end{equation}
It counts the $\kappa=0$ special breakpoint graphs whose underlying permutation has precisely this nontrivial cycle type.

\begin{thm}[Fixed cycle type]\label{th:fixed_cycle_type}
For every fixed partition $\lambda\vdash m$ with $q=q(\lambda)$ nontrivial cycles, $A_{n,\lambda}$ is a polynomial in $n$ of degree at most
\[
m+q.
\]
More precisely, let $\mathcal T_\lambda$ be the finite set of pointed local types $(\varepsilon,\sigma)$ from the proof of Theorem~\ref{th:fixed_u_poly} for which $\sigma$ has exactly $d_j$ cycles of length $j+1$ and $\xi_{\mathrm{loc}}(\varepsilon,\sigma)=m$.  Then
\begin{equation}\label{eq:fixed_cycle_type_formula}
A_{n,\lambda}
=
\frac{n}{m+q}
\sum_{(\varepsilon,\sigma)\in\mathcal T_\lambda}
\binom{n+t(\varepsilon)/2-1}{m+q-1}.
\end{equation}
In particular,
\begin{equation}\label{eq:fixed_cycle_type_divisibility}
n(n-1)\cdots(n-m+1)\ \mid\ A_{n,\lambda}
\end{equation}
as a polynomial in $n$.
\end{thm}

\begin{proof}
For the prescribed cycle type, every contributing permutation has support size
\[
r=\sum_{j\ge1}(j+1)d_j=m+q,
\]
so there is no longer a sum over support sizes in the proof of Theorem~\ref{th:fixed_u_poly}.  For a fixed pointed local type $(\varepsilon,\sigma)$, equation~\eqref{eq:local_embedding_count} gives
\[
n\binom{n+t(\varepsilon)/2-1}{r-1}
\]
pointed embeddings into $\{1,\ldots,2n\}$.  Each unpointed permutation is represented exactly $r$ times, according to the distinguished support point.  Since $A_{n,\lambda}$ is the unnormalized count, no factor $1/(2n)$ remains, and summing over the admissible local types gives~\eqref{eq:fixed_cycle_type_formula}.

Every binomial term has degree $r-1$, and the prefactor $n$ therefore gives degree at most $r=m+q$.  Moreover, the explicit factor $n$ gives a zero at $n=0$, while $m$ gray-red cycles are impossible for $1\le n<m$.  Hence the polynomial vanishes at $0,1,\ldots,m-1$, proving~\eqref{eq:fixed_cycle_type_divisibility}.
\end{proof}

Thus, after factoring out the universal falling factorial, every fixed cycle type has the form
\begin{equation}\label{eq:fixed_cycle_type_quotient}
A_{n,\lambda}
=n(n-1)\cdots(n-m+1)\,P_\lambda(n),
\qquad
\deg P_\lambda\le q.
\end{equation}
The number of nontrivial cycles, rather than merely the transposition length, controls the remaining degree.  This gives a refinement of Theorem~\ref{th:fixed_u_poly}: after a weight specialization, $B_m(n)$ is a weighted sum over the finitely many partitions $\lambda\vdash m$ of the polynomials $A_{n,\lambda}/(2n)$.

\paragraph{One nontrivial cycle.}
The case $q=1$, namely $\lambda=(m)$, has an especially simple closed form.  Equivalently, among the black-gray cycle variables we retain the monomial
\[
s_1^{2n-m-1}s_{m+1}.
\]
If the Catalan weights are kept formal before the final specialization, this is precisely the contribution conventionally described as the coefficient of $C_m$: the unique nontrivial permutation cycle has length $m+1$ and hence receives the weight $C_m$.

For convenience write
\begin{equation}\label{eq:Anm_def}
A_{n,m}:=A_{n,(m)}
=
2n\,\coef{x^n u^m s_1^{2n-m-1}s_{m+1}}\,\G_0.
\end{equation}

\begin{prop}[One nontrivial cycle]\label{prop:A135065}
For all $1\le m\le n$,
\begin{equation}\label{eq:Anm_closed}
A_{n,m}=n^2\binom{n-1}{m-1}.
\end{equation}
Consequently the triangular array $(A_{n,m})_{1\le m\le n}$ begins
\[
\begin{array}{c|ccccc}
n=1&1\\
n=2&4&4\\
n=3&9&18&9\\
n=4&16&48&48&16\\
n=5&25&100&150&100&25
\end{array}
\]
and, after the reindexing $N=n-1$, $K=m-1$, its entries are
\[
T(N,K):=(N+1)^2\binom{N}{K}.
\]
Its bivariate generating function is therefore
\begin{equation}\label{eq:A135065_bivariate}
\sum_{n\ge1}\sum_{m=1}^n A_{n,m}z^{n-1}y^{m-1}
=
\frac{1+z+zy}{(1-z-zy)^3}.
\end{equation}
\end{prop}

The same triangular array, under the reindexing $N=n-1$, $K=m-1$, is tabulated in the OEIS as sequence A135065~\cite{OEIS-A135065}.

\begin{proof}
Let $\mathcal L_1$ denote the part of $\G_0$ that is homogeneous of degree one in the variables $s_2,s_3,\ldots$.  By definition,
\begin{equation}\label{eq:L_onecycle}
\mathcal L_1
:=
\sum_{n\ge1}\sum_{m=1}^n
\frac{A_{n,m}}{2n}
 x^n u^m s_1^{2n-m-1}s_{m+1}.
\end{equation}
Differentiate the $\kappa=0$ equation with respect to $s_1$, as in the proof of Theorem~\ref{th:G0_PDE}:
\[
\frac{\partial\G_0}{\partial s_1}
=
x\sum_{i\ge1}i\left(us_{i+1}+\sum_{j=1}^i s_js_{i+1-j}\right)
\frac{\partial\G_0}{\partial s_i}.
\]
Taking the component homogeneous of degree one in $s_2,s_3,\ldots$, and writing $s=s_1$, gives
\begin{equation}\label{eq:L_onecycle_PDE}
(\mathcal L_1)_s
=
x\left[
 s^2(\mathcal L_1)_s
 +\sum_{i\ge2}i\bigl(us_{i+1}+2ss_i\bigr)(\mathcal L_1)_{s_i}
\right].
\end{equation}
Indeed, in the convolution $\sum_{j=1}^i s_js_{i+1-j}$ only the two terms with one factor $s_1$ can contribute to this homogeneous component when $i\ge2$.

Extract the coefficient of
\[
x^n u^m s^{2n-m-2}s_{m+1}
\]
from~\eqref{eq:L_onecycle_PDE}.  With the convention $A_{n,m}=0$ outside $1\le m\le n$, we obtain
\begin{equation}\label{eq:Anm_recurrence}
\frac{2n-m-1}{2n}A_{n,m}
=
\frac{mA_{n-1,m-1}+(2n+m-1)A_{n-1,m}}{2(n-1)}
\qquad(n\ge2).
\end{equation}
The initial value is $A_{1,1}=1$, from the leading term $\G_0=us_2x/2+\cdots$.

We now verify that~\eqref{eq:Anm_closed} satisfies the recurrence.  Using
\[
\binom{n-2}{m-2}=\frac{m-1}{n-1}\binom{n-1}{m-1},
\qquad
\binom{n-2}{m-1}=\frac{n-m}{n-1}\binom{n-1}{m-1},
\]
the right-hand side of~\eqref{eq:Anm_recurrence}, after substituting
$A_{n,m}=n^2\binom{n-1}{m-1}$, reduces to
\[
\frac12\left[m(m-1)+(2n+m-1)(n-m)\right]
\binom{n-1}{m-1}.
\]
Since
\[
m(m-1)+(2n+m-1)(n-m)=n(2n-m-1),
\]
this is exactly the left-hand side.  The initial condition therefore proves~\eqref{eq:Anm_closed}.  Finally, summing the binomial formula gives
\[
\sum_{n\ge1}\sum_{m=1}^n A_{n,m}z^{n-1}y^{m-1}
=
\sum_{N\ge0}(N+1)^2\bigl(z(1+y)\bigr)^N
=
\frac{1+z+zy}{(1-z-zy)^3},
\]
which proves~\eqref{eq:A135065_bivariate} and completes the proof.
\end{proof}

For $m=1$, Proposition~\ref{prop:A135065} gives $A_{n,(1)}=n^2$, so the square sequence is the one-nontrivial-cycle slice with a transposition.  For $m=2$, there are only two partitions.  The one-cycle contribution is
\[
A_{n,(2)}=n^2(n-1),
\]
while the full Catalan specialization gives
\[
2nB_2(n)=\frac{n^2(n-1)(5n+2)}6.
\]
Since the Catalan weights are $C_2=2$ and $C_1^2=1$, the complementary two-transposition contribution is
\begin{equation}\label{eq:A11_closed}
A_{n,(1,1)}=\frac56 n^2(n-1)(n-2).
\end{equation}
Thus the fixed-$m=2$ polynomial is completely resolved by cycle type.

\paragraph{A triangular recurrence in the number of nontrivial cycles.}
The $\kappa=0$ differential equation~\eqref{eqn:pde_kappa_0} gives a systematic way to continue this refinement.  For $q\ge0$, let $\mathcal L_q$ denote the part of $\G_0$ homogeneous of degree $q$ in $s_2,s_3,\ldots$, with $\mathcal L_0:=0$.  Thus
\[
\G_0=\sum_{q\ge1}\mathcal L_q.
\]
Collecting homogeneous degrees in the differentiated $\kappa=0$ equation gives
\begin{align}
(\mathcal L_q)_s
=x\Bigg[&s^2(\mathcal L_q)_s
+\sum_{i\ge2}i\bigl(us_{i+1}+2ss_i\bigr)(\mathcal L_q)_{s_i}
+us_2(\mathcal L_{q-1})_s \notag\\
&+\sum_{i\ge3}i\left(\sum_{j=2}^{i-1}s_js_{i+1-j}\right)(\mathcal L_{q-1})_{s_i}
\Bigg],
\label{eq:Lq_triangular}
\end{align}
where $s:=s_1$.  Thus the coefficients with $q$ nontrivial cycles satisfy linear recurrences whose inhomogeneous terms involve only the already determined $(q-1)$-cycle coefficients.  The exceptional M\"obius ladders provide the $s=0$ boundary terms when needed.

As an illustration, transposition length $m=3$ has the three cycle types $(3)$, $(2,1)$, and $(1,1,1)$.  Applying~\eqref{eq:Lq_triangular}, together with the one-cycle formula and~\eqref{eq:A11_closed}, gives the following closed forms.

\begin{corol}[All cycle types at transposition length three]\label{cor:m3_cycle_types}
For $n\ge3$,
\begin{align}
A_{n,(3)}
&=\frac12 n^2(n-1)(n-2),\label{eq:A3_closed}\\
A_{n,(2,1)}
&=\frac16 n^2(n-1)(n-2)(5n-13),\label{eq:A21_closed}\\
A_{n,(1,1,1)}
&=\frac1{180}n^2(n-1)(n-2)
\bigl(44n^2-243n+343\bigr).\label{eq:A111_closed}
\end{align}
\end{corol}

\begin{proof}
The first identity is Proposition~\ref{prop:A135065} with $m=3$.  Set
\[
Q_n:=A_{n,(2,1)},\qquad R_n:=A_{n,(1,1,1)}.
\]
Extracting respectively the coefficients of
\[
x^nu^3s^{2n-6}s_2s_3
\qquad\text{and}\qquad
x^nu^3s^{2n-7}s_2^3
\]
from~\eqref{eq:Lq_triangular} yields
\begin{align}
\frac{2n-5}{2n}Q_n
&=\frac{(2n+3)Q_{n-1}+(2n-5)A_{n-1,(2)}
+4A_{n-1,(1,1)}+8A_{n-1,(3)}}{2(n-1)},
\label{eq:Qn_rec}\\
\frac{2n-6}{2n}R_n
&=\frac{(2n+4)R_{n-1}+(2n-6)A_{n-1,(1,1)}+3Q_{n-1}}{2(n-1)}.
\label{eq:Rn_rec}
\end{align}
The first recurrence starts with $Q_3=6$.  For the second, $R_3=1$ is the $M_{12}$ M\"obius-ladder boundary term, and~\eqref{eq:Rn_rec} applies for $n\ge4$.  Substitution of Proposition~\ref{prop:A135065}, equation~\eqref{eq:A11_closed}, and the expressions in~\eqref{eq:A21_closed}--\eqref{eq:A111_closed} verifies the two recurrences identically, proving the result by induction.
\end{proof}

For example, at $n=3$ the three coefficients are
\[
A_{3,(1,1,1)}=1,\qquad A_{3,(2,1)}=6,\qquad A_{3,(3)}=9.
\]
After Catalan specialization their weighted sum is
\[
1\cdot C_1^3+6\cdot C_1C_2+9\cdot C_3
=1+12+45=58,
\]
which is the $n=3$ diagonal coefficient appearing in Subsection~\ref{sec:trees}.  More generally, Theorem~\ref{th:fixed_cycle_type} and the triangular recurrence~\eqref{eq:Lq_triangular} reduce every coefficient at fixed transposition length to a finite collection of polynomial cycle-type contributions.

\section{Application to random Gaussian-state entanglement}\label{sec:physics_application}

We now return to the second motivation mentioned in the Introduction. This section contains no new breakpoint-graph proofs: its purpose is to give a dictionary between the mathematical generating functions above and a large-system expansion for entanglement of Haar-random Gaussian states, and then to record the coefficient-level physical consequences of Propositions~\ref{prop:G1_signed_Catalan} and~\ref{prop:Catalan_surface_closed_form}. The underlying Page-curve problem and its moment expansion are developed in~\cite{Iosue2023,Youm2025}.

\subsection{Physical setup and the breakpoint-graph dictionary}

Consider $N$ bosonic modes, initially equally squeezed with squeezing parameter $s>0$, followed by a Haar-random passive linear-optical transformation $U\in U(N)$. Let a subsystem contain $k=\lfloor rN\rfloor$ modes, with $0\leq r\leq1$, and let $S_2(U;s,r)$ denote its R\'enyi-$2$ entanglement entropy. A convenient expansion~\cite{Iosue2023} is
\begin{equation}\label{eq:physics_S2_series}
S_2(U;s,r)
=
\sum_{\ell\geq1}\frac{\vartheta^\ell}{2\ell}
\left(k-\operatorname{Tr}W^\ell\right),
\qquad
\vartheta:=\tanh^2(2s),
\end{equation}
where $W$ is the positive contraction obtained from the subsystem block of $UU^T$, namely $W=\Pi U U^T \Pi \bar{U} \bar{U}^T \Pi$, with $\Pi$ the projection onto the chosen $k$-mode subsystem. For fixed $\ell$, its expected moment has an expansion
\begin{equation}\label{eq:physics_moment_expansion}
\mathbb E\operatorname{Tr}W^\ell
=Nf_\ell(r)+H_\ell(r)+o(1)
\qquad(N\to\infty).
\end{equation}
The leading and constant terms are polynomials in the subsystem fraction $r$. Let $a_1^{(\ell)}$ be the coefficient of $r^{2\ell}$ in $f_\ell(r)$ and $a_0^{(\ell)}$ the coefficient of $r^{2\ell}$ in $H_\ell(r)$.

The unitary-Weingarten expansion of these extremal coefficients has the form
\begin{equation}\label{eq:physics_weingarten_sum}
a_\kappa^{(\ell)}
=
\sum_{\substack{\pi\in S_{2\ell}:\\
\xi^{(\ell)}(\pi)-|\pi|=\kappa}}
(-1)^{|\pi|}
\prod_{c\in\operatorname{cyc}(\pi)}C_{|c|-1},
\qquad \kappa\in\{0,1\},
\end{equation}
where $|\pi|:=2\ell-\#(\pi)$ is the transposition length and $\xi^{(\ell)}(\pi)$ is the second free-index-loop count in the cyclic contraction. Under the special-breakpoint-graph correspondence, $\#(\pi)$ and $\xi^{(\ell)}(\pi)$ are precisely the black-gray and gray-red cycle counts. Thus~\eqref{eq:physics_weingarten_sum} is exactly the permutation-length form $\kappa=\xi-|\pi|$ of the graph parameter introduced in Section~\ref{sec:breakpoint}.

The relation with our generating functions is particularly direct. Since setting $u=-1$ contributes the sign $(-1)^{\xi(\pi)}$ while the physical sum uses $(-1)^{|\pi|}$, the stratum identity $\xi(\pi)-|\pi|=\kappa$ gives
\begin{equation}\label{eq:physics_graph_interface}
2\ell\,\coef{x^\ell}\,
\G_\kappa(x,-1,C_0,C_1,C_2,\ldots)
=(-1)^\kappa a_\kappa^{(\ell)}.
\end{equation}
Thus the two orders selected by the physical expansion have an intrinsic topological meaning:
\[
\kappa=1\quad\Longleftrightarrow\quad \Sigma_G\cong S^2,
\qquad
\kappa=0\quad\Longleftrightarrow\quad \Sigma_G\cong\mathbb{RP}^2.
\]
In this sense the leading and constant extremal coefficients probe respectively the spherical and projective-plane special-breakpoint-graph strata.

\subsection{\texorpdfstring{R\'enyi-$2$}{Renyi-2} moment polynomials and the constant correction}

Combining~\eqref{eq:physics_graph_interface} with the signed Catalan evaluations proved in Sections~\ref{sec:kappa=1} and~\ref{sec:kappa=0} immediately gives
\begin{equation}\label{eq:physics_a10_from_graphs}
a_1^{(\ell)}=(-1)^{\ell-1}C_{\ell-1},
\qquad
a_0^{(\ell)}=(-1)^\ell4^{\ell-1}.
\end{equation}
The first identity is Proposition~\ref{prop:G1_signed_Catalan}; the second is Proposition~\ref{prop:Catalan_surface_closed_form}. No additional graph enumeration is needed in this section.

The subsystem symmetry $r\leftrightarrow1-r$ and the polynomial support of the moment coefficients~\cite{Iosue2023} then turn the two extremal coefficients into the full leading and constant moment terms. Writing
\begin{equation}\label{eq:physics_R_def}
R:=r(1-r),
\end{equation}
one obtains
\begin{equation}\label{eq:physics_Gell_Hell}
G_\ell(r):=r-f_\ell(r)
=\sum_{j=1}^{\ell}C_{j-1}R^j,
\qquad
H_\ell(r)=4^{\ell-1}R^\ell.
\end{equation}
The first formula is the Catalan polynomial determined by the spherical coefficient, while the second shows that the projective-plane evaluation fixes the entire constant moment correction.

Taking expectations in~\eqref{eq:physics_S2_series} therefore gives
\begin{equation}\label{eq:physics_entropy_asymptotic}
\mathbb E S_2(U;s,r)
=N\beta_2(s,r)-\lambda_2(s,r)+o(1),
\end{equation}
with the coefficient-level representations
\begin{equation}\label{eq:physics_beta2_lambda2_series}
\beta_2(s,r)
=\sum_{\ell\geq1}\frac{\vartheta^\ell}{2\ell}G_\ell(r),
\qquad
\lambda_2(s,r)
=\sum_{\ell\geq1}\frac{\vartheta^\ell}{2\ell}H_\ell(r),
\qquad
\vartheta=\tanh^2(2s).
\end{equation}
The first series is the R\'enyi-$2$ Page-curve expansion of~\cite{Iosue2023}, now with its full moment polynomials recovered from the spherical special-breakpoint-graph evaluation. Its closed-form summation is carried out in recent related work~\cite{AltPageCurves}; since that final resummation introduces no further breakpoint-graph input, we do not reproduce it here.

For the constant term, the projective-plane formula in~\eqref{eq:physics_Gell_Hell} resums immediately:
\begin{equation}\label{eq:physics_lambda2_closed}
\lambda_2(s,r)
=-\frac18\log\!\left(
1-4r(1-r)\tanh^2(2s)
\right).
\end{equation}
Thus the breakpoint-graph calculation determines the complete $O(N)$ and $O(1)$ moment data; in the $O(1)$ case, the passage from the projective-plane coefficients to the entropy correction is the elementary logarithmic generating-function sum above.

\subsection{Integer R\'enyi index}

The reduction for integer R\'enyi index $\alpha\geq2$ in~\cite{Youm2025} expresses the leading and constant finite-size terms through the same sequences $G_\ell(r)$ and $H_\ell(r)$ in~\eqref{eq:physics_Gell_Hell}. Consequently, the spherical and projective-plane graph evaluations supply the corresponding moment input for every integer $\alpha\geq2$ without any further breakpoint-graph calculation. Writing
\begin{equation}\label{eq:physics_Renyi_alpha_asymptotic}
\mathbb E S_\alpha(U;s,r)
=N\beta_\alpha(s,r)-\lambda_\alpha(s,r)+o(1),
\end{equation}
recent related work~\cite{AltPageCurves} performs the remaining Gaussian-state resummations and gives closed forms for both $\beta_\alpha(s,r)$ and $\lambda_\alpha(s,r)$. We do not reproduce those formulas here: the role of the present paper is to supply, through the two topological breakpoint-graph strata, the combinatorial moment coefficients that enter that reduction. The R\'enyi-$2$ constant~\eqref{eq:physics_lambda2_closed} is retained above because it is an immediate one-line resummation of the projective-plane sequence $H_\ell$.

\subsection{Relation to recent Page-curve derivations}\label{sec:physics_alternative}

Recent related work~\cite{AltPageCurves} takes the same entanglement problem from the moment expansion to closed physical formulas. For the R\'enyi-$2$ Page curve it gives two complementary derivations: one uses the Jacobi $\beta=1$ ensemble associated with the singular values of the subsystem block of $UU^T$ and extracts the extensive and constant terms from its resolvent, while the other starts directly from the infinite Page-curve series of~\cite{Iosue2023} and sums it without random-matrix resolvents. The same work also gives closed leading and constant terms for every integer R\'enyi index $\alpha\geq2$.

The comparison with the present approach is clearest at the level of the moment coefficients in~\eqref{eq:physics_Gell_Hell}. The direct-series derivation of~\cite{AltPageCurves} encounters the same Catalan polynomial $G_\ell(r)$ that here follows from the spherical $\kappa=1$ evaluation together with the symmetry and support properties of~\cite{Iosue2023}. In its spectral derivation, the constant resolvent term yields the same sequence $H_\ell(r)$ that here comes from the signed $\kappa=0$ Catalan evaluation in the projective-plane class. Thus the distinctive contribution of the breakpoint-graph route is the combinatorial and topological origin of these moment coefficients, while the final entropy resummations may be carried out by the methods of~\cite{AltPageCurves}. That work also treats Gaussian hafnian moments and anticoncentration, as well as unequal input squeezing; those questions are outside the scope of the present paper.

\section{Conclusion and outlook}\label{sec:extensions}
We developed a self-contained combinatorial treatment of the two largest strata of special breakpoint graphs. The canonical surface carried by a graph has Euler characteristic $\chi=\kappa+1$, so $\kappa=1$ is spherical and $\kappa=0$ is carried by the projective plane. This topology is reflected directly in the recurrences: noncrossing spherical graphs decompose into three smaller graphs, whereas projective-plane graphs are reduced by removing a rooted black-gray $1$-cycle, with odd M\"obius ladders providing the boundary family.

The resulting multivariate equations admit unexpectedly small reductions. The $\kappa=1$ shift field is conjugate to translation of the cycle-weight generating series and gives a canonical one-parameter reduction through every target sequence; on the signed Catalan orbit this yields the explicit coefficient formula~\eqref{eq:G1_signed_Catalan_coeff}. For $\kappa=0$, the relevant derivations form $\mathfrak{sl}_2$, so every target lies on an orbit of dimension at most three, with a stabilizer forcing a further dimension drop. At $u=-1$ the Catalan point is fixed by $V$; its two-dimensional orbit is projective/Riccati and the restricted generating function has the closed form~\eqref{eq:Catalan_surface_closed_form}.

A second finiteness phenomenon occurs coefficientwise in the $\kappa=0$ class. Fixing the number of gray-red cycles bounds the permutation support and forces polynomiality in $n$, while fixing the nontrivial cycle type sharpens the degree bound and produces a triangular recurrence by the number of nontrivial cycles. The one-cycle family has the explicit binomial form~\eqref{eq:Anm_closed}, and the diagonal Catalan specialization recovers the total node-depth generating function for noncrossing trees.

The final application section translates two of the intrinsic Catalan evaluations back to the random Gaussian-state problem that helped motivate them. The unitary-Weingarten expansion selects exactly the same $\kappa=1$ and $\kappa=0$ sums, so the spherical and projective-plane formulas give the full leading and constant moment polynomials entering the R\'enyi-$2$ Page-curve expansion. The projective-plane polynomial resums immediately to the logarithmic constant correction, while the same spherical and projective-plane moment data feed the existing reduction for every integer R\'enyi index. Closed entropy resummations beyond this immediate constant term are left to the related physics treatment~\cite{AltPageCurves}. Thus the physical application uses the combinatorial theory as an input rather than altering its mathematical organization.

Several directions suggested by these results remain open.
The two topological strata studied here already exhibit two rather different mechanisms for reducing an infinite family of cycle variables. For $\kappa=1$, Proposition~\ref{prop:G1_shift_flow} shows that every target weight sequence lies on a canonical one-parameter shift orbit, while Proposition~\ref{prop:G1_signed_Catalan} shows that the Catalan orbit becomes explicitly solvable at $u=-1$. A natural next question within the spherical class is to characterize other initial series $f$ for which the reduced equation~\eqref{eq:G1_reduced_generic} has an algebraic, D-finite, or otherwise explicit solution at the target point.

For $\kappa=0$, the corresponding reduction problem is controlled by the $\mathfrak{sl}_2$ action of Theorem~\ref{th:sl2}. Equation~\eqref{eq:stabilizer_ODE} reduces the search for two-dimensional orbits to a first-order differential equation for one weight generating series. The Catalan fixed point is the simplest nontrivial example, and the projective coordinates of Proposition~\ref{prop:Catalan_projective_coordinates} make its orbit completely explicit. It would be useful to classify other formal solutions of the stabilizer equation, and more generally finite-dimensional families tangent to the relevant vector fields, for which the restricted PDE can be solved in closed form.

The preceding results suggest a practical reduction strategy for other multivariate graph equations.  First isolate the derivations through which the cycle variables enter.  Next compute their Lie closure.  A finite-dimensional Lie closure produces canonical finite-dimensional formal orbits on which the PDE closes.  This is closely analogous to the Lie--Scheffers viewpoint, where a differential system whose time-dependent vector field takes values in a finite-dimensional Lie algebra is controlled by the corresponding finite-dimensional group action~\cite{CarinenaDeLucas2011}.  Here the ambient space is the infinite space of cycle weights, and the role of the group orbit is played by a finite-dimensional formal invariant family.  If the Lie algebra is larger, one may instead search for a finite-dimensional family $Q(z;\theta_1,\ldots,\theta_d)$ satisfying tangency relations
\begin{equation}\label{eq:tangency_general}
X_aQ=\sum_{b=1}^d B_{ab}(\theta)\frac{\partial Q}{\partial\theta_b}
\end{equation}
for every derivation $X_a$ occurring in the original equation.  Equations such as~\eqref{eq:stabilizer_ODE} are the first nontrivial instance of this principle.  This viewpoint separates two questions that are easy to conflate: whether a finite-dimensional reduction exists, and whether the resulting finite-dimensional equation can be solved in closed form.

The coefficient formulas in Subsections~\ref{sec:polynomiality} and~\ref{sec:cycle_type_refinement} suggest a complementary finite problem. Theorem~\ref{th:fixed_cycle_type} shows that for a partition $\lambda\vdash m$ with $q$ nontrivial cycles, the contribution $A_{n,\lambda}$ is a falling factorial times a polynomial of degree at most $q$. Proposition~\ref{prop:A135065} solves the one-cycle case, and Corollary~\ref{cor:m3_cycle_types} resolves all cycle types at transposition length three. Determining exact degrees, leading coefficients, and systematic factorizations of $P_\lambda(n)$ for general $\lambda$ is a concrete problem entirely within the projective-plane class. Direct bijections explaining the simple one-cycle formula would be especially valuable.

Finally, Corollary~\ref{cor:noncrossing_trees} identifies the Catalan-weighted diagonal of the $\kappa=0$ series with the classical total path-length statistic for noncrossing trees studied by Deutsch and Noy~\cite{DeutschNoy2002}. The proof here is analytic, through the shift-flow PDE and Lagrange inversion. A direct correspondence, ideally refined by the black-gray cycle structure, remains a natural combinatorial problem and could reveal why this established tree statistic appears in both the invariant-orbit and diagonal reductions.

The physical application also suggests a topological direction for future work. The leading and constant terms select the spherical and projective-plane strata, and lower orders of the same large-system expansion should naturally involve strata with $\kappa<0$. A detailed unrestricted treatment now in preparation~\cite{UnrestrictedBreakpointGraphs} develops an exact character expansion together with local deletion kernels for all four deletion cases. After repairing the black-red Hamiltonian cycle following nonadjacent deletion, both nonadjacent contributions close through an exchange-symmetric two-point series that tracks the two opened paths; the finite pairing kernels can be organized in Brauer-diagram language~\cite{Brauer1937}. Thus local closure itself is no longer the obstruction to extracting lower fixed-$\kappa$ strata.

At uniform cycle weight, that unrestricted treatment now also resolves the two ingredients that initially obstructed the first lower stratum. The symmetric length-marked statistic $E_+^{(y)}$ admits a finite two-distinguished-point character formula, while orientability is detected directly by parity preservation of the encoding permutation. Consequently the stratum $\kappa=-1$, for which $\chi=\kappa+1=0$, separates explicitly into torus and Klein-bottle sectors. The remaining problem is therefore no longer to construct the missing marked states or detect orientability, but to simplify the resulting finite formulas and extract effective closed recurrences. This provides the natural next step toward usable fixed-$\kappa$ enumeration below the projective-plane stratum.

\section*{Acknowledgements}

The authors acknowledge the use of OpenAI's ChatGPT as an assistive tool in the preparation of this manuscript, including literature discovery, exploratory derivations, symbolic and computational checks, organization of the exposition, and language and LaTeX editing. All mathematical statements, proofs, citations, assessments of prior work, and final wording were reviewed by the authors, who assume full responsibility for the content of the manuscript. J.T.I.~and A.V.G.\ acknowledge support from the U.S.~Department of Energy, Office of Science, Accelerated Research in Quantum Computing, Fundamental Algorithmic Research toward Quantum Utility (FAR-Qu). J.T.I.~and A.V.G.\ were also supported in part by ARL (W911NF-24-2-0107), ONR MURI, NSF QLCI (award No.~OMA-2120757), NQVL:QSTD:Design:FTL, DoE ASCR Quantum Testbed Pathfinder program (award No.~DE-SC0024220), NSF STAQ program, and AFOSR MURI. J.T.I.~and A.V.G.\ also acknowledge support from the U.S.~Department of Energy, Office of Science, National Quantum Information Science Research Centers, Quantum Systems Accelerator (award No.~DE-SCL0000121). J.T.I.'s contributions were made during his time at UMD.


\begin{thebibliography}{10}

\bibitem{Alexeev2017hultman}
N.~Alexeev, A.~Pologova, and M.~A.~Alekseyev,
\newblock Generalized Hultman numbers and cycle structures of breakpoint graphs,
\newblock \emph{Journal of Computational Biology} 24 (2017), no.~2, 93--105. \doi{10.1089/cmb.2016.0190}.

\bibitem{Alekseyev2007}
M.~A.~Alekseyev and P.~A.~Pevzner,
\newblock Colored de Bruijn graphs and Genome Halving Problem,
\newblock \emph{IEEE/ACM Transactions on Computational Biology and Bioinformatics} 4 (2007), no.~1, 98--107. \doi{10.1109/TCBB.2007.1002}.

\bibitem{Alexeev2016top}
N.~Alexeev, P.~Avdeyev, and M.~A.~Alekseyev,
\newblock Comparative genomics meets topology: a novel view on genome median and halving problems,
\newblock \emph{BMC Bioinformatics} 17 (2016), Suppl.~14, 213--223. \doi{10.1186/s12859-016-1263-7}.

\bibitem{OEIS-A062236}
OEIS Foundation Inc.,
\newblock The On-Line Encyclopedia of Integer Sequences, sequence A062236.

\bibitem{OEIS-A135065}
OEIS Foundation Inc.,
\newblock The On-Line Encyclopedia of Integer Sequences, sequence A135065.

\bibitem{Noy1998}
M.~Noy,
\newblock Enumeration of noncrossing trees on a circle,
\newblock \emph{Discrete Mathematics} 180 (1998), no.~1--3, 301--313. \doi{10.1016/S0012-365X(97)00121-0}.

\bibitem{DeutschNoy2002}
E.~Deutsch and M.~Noy,
\newblock Statistics on non-crossing trees,
\newblock \emph{Discrete Mathematics} 254 (2002), no.~1--3, 75--87. \doi{10.1016/S0012-365X(01)00366-1}.

\bibitem{Ore1933}
{\O}.~Ore,
\newblock Theory of non-commutative polynomials,
\newblock \emph{Annals of Mathematics} 34 (1933), no.~3, 480--508. \doi{10.2307/1968173}.

\bibitem{ChyzakSalvy1998}
F.~Chyzak and B.~Salvy,
\newblock Non-commutative elimination in Ore algebras proves multivariate identities,
\newblock \emph{Journal of Symbolic Computation} 26 (1998), no.~2, 187--227. \doi{10.1006/jsco.1998.0207}.

\bibitem{Stanley1980}
R.~P.~Stanley,
\newblock Differentiably finite power series,
\newblock \emph{European Journal of Combinatorics} 1 (1980), no.~2, 175--188. \doi{10.1016/S0195-6698(80)80051-5}.

\bibitem{CarinenaDeLucas2011}
J.~F.~Cari\~nena and J.~de Lucas,
\newblock Lie systems: theory, generalisations, and applications,
\newblock \emph{Dissertationes Mathematicae} 479 (2011), 1--162. \doi{10.4064/dm479-0-1}.

\bibitem{Brauer1937}
R.~Brauer,
\newblock On algebras which are connected with the semisimple continuous groups,
\newblock \emph{Annals of Mathematics} 38 (1937), no.~4, 857--872. \doi{10.2307/1968843}.

\bibitem{Iosue2023}
J.~T.~Iosue, A.~Ehrenberg, D.~Hangleiter, A.~Deshpande, and A.~V.~Gorshkov,
\newblock Page curves and typical entanglement in linear optics,
\newblock \emph{Quantum} 7 (2023), 1017. \doi{10.22331/q-2023-05-23-1017}.

\bibitem{Youm2025}
J.~Youm, J.~T.~Iosue, A.~Ehrenberg, Y.-X.~Wang, and A.~V.~Gorshkov,
\newblock Average R\'enyi entanglement entropy in Gaussian boson sampling,
\newblock \emph{Physical Review Research} 7 (2025), 023125. \doi{10.1103/PhysRevResearch.7.023125}.

\bibitem{AltPageCurves}
L.~Shou, A.~Ehrenberg, Y.-X.~Wang, J.~T.~Iosue, and A.~V.~Gorshkov, \emph{Anticoncentration and entanglement in Gaussian boson sampling},
\newblock manuscript, 2026.

\bibitem{UnrestrictedBreakpointGraphs}
M.~A.~Alekseyev,
\newblock Unrestricted cycle enumeration and two-point harmonic analysis of special breakpoint graphs,
\newblock in preparation, 2026.

\end{thebibliography}
\end{document}